\documentclass[reqno]{amsart}

\usepackage{amsmath,times}
\usepackage{amsthm,amssymb}
\usepackage[all]{xy}
\usepackage{enumitem}

\usepackage{latexsym}
\usepackage{stackrel}

\usepackage{graphicx}
\usepackage{color}
\usepackage{float}
\usepackage{array}
\newcolumntype{L}{>{$}l<{$}} 

\usepackage{comment}

\usepackage{hyperref}
\hypersetup{
    colorlinks=true, 
    linktoc=all,     
        linkcolor=blue,  
}

\theoremstyle{definition}

\newtheorem{theorem}{Theorem}[section]
\newtheorem{lemma}[theorem]{Lemma}

\newtheorem{definitionlemma}[theorem]{Definition-Lemma}

\newtheorem{proposition}[theorem]{Proposition}

\theoremstyle{definition}
\newtheorem{example}[theorem]{Example}
\newtheorem{definition}[theorem]{Definition}
\newtheorem{corollary}[theorem]{Corollary}
\newtheorem{remark}[theorem]{Remark}

\numberwithin{equation}{section}

\newenvironment{sketchproof}{\paragraph{\it Sketch Proof.}}{\hfill$\square$}

\newcommand{\C}{\mathbb{C}}

\newcommand{\N}{\mathbb{N}}
\renewcommand{\P}{\mathbb{P}}
\newcommand{\Q}{\mathbb{Q}}
\newcommand{\R}{\mathbb{R}}
\newcommand{\Z}{\mathbb{Z}}

\def\d{\mbox{d}}
\def\vol{\mbox{vol}}

\DeclareMathOperator{\GL}{GL}
\newcommand{\Grass}{\operatorname{Grass}}
\newcommand{\Hom}{\operatorname{Hom}}
\newcommand{\Ker}{\operatorname{Ker}}

\newcommand{\rank}{\operatorname{rk}}
\newcommand{\sub}{\operatorname{sub}}

\newcommand{\cC}{\mathcal{C}}
\newcommand{\cE}{\mathcal{E}}
\newcommand{\cH}{\mathcal{H}}

\newcommand{\cO}{\mathcal{O}}
\newcommand{\cQ}{\mathcal{Q}}
\newcommand{\cR}{\mathcal{R}}
\newcommand{\cS}{\mathcal{S}}

\title[{H}odge-{R}iemann Relations in the Linear and K\"ahler Cases]{{H}odge-{R}iemann  Relations for {S}chur Classes\\ in the Linear and K\"ahler Cases}

\author{Julius Ross and Matei Toma}

\address{Department of Mathematics, Statistics, and Computer Science, University of Illinois at Chicago, 322 Science and Engineering Offices (M/C 249), 851 S. Morgan Street, Chicago, IL 60607
 }
\email{juliusro@uic.edu}

\address{Universit\'e de Lorraine, CNRS, IECL, F-54000 Nancy, France
 }
\email{matei.toma@univ-lorraine.fr}

\keywords{ 14C17, 14J60, 32J27, 52A40}
\date{\today}

\begin{document}
\begin{abstract}We prove a version of the Hodge-Riemann bilinear relations for Schur polynomials of K\"ahler forms and for Schur polynomials of positive forms on a complex vector space.
\end{abstract}

\maketitle

\section{Introduction}

Let $(X,\omega)$ be a compact K\"ahler manifold of dimension $d\ge2$ and let $H^{1,1}_\R(X)$ be its real Dolbeault cohomology group in bi-degree $(1,1)$. The classical Hodge-Riemann bilinear relations imply that the quadratic form 
$$Q_{\{\omega\}^{d-2}}:H^{1,1}_\R(X)\to\R, \ \{\alpha\}\mapsto\int_X\alpha^2\wedge\omega^{d-2}$$
is non-degenerate of signature $(1,h^{1,1}-1)$.  We will  summarize this by saying that the class $\{\omega\}^{d-2}$ has the {\em Hodge-Riemann property}. 

When $X$ is a complex torus we may take parallel representatives in the classes $\{\omega\}$ and $\{\alpha\}$ and the above statement reduces to its ``linear version'', namely that if $E$ is a $d$-dimensional complex vector space and $\omega$ is a strictly positive $(1,1)$-form on $E$ then the quadratic form 
$$Q_{\omega^{d-2}}:\bigwedge^{1,1}_\R E^*\to\R, \ \alpha\mapsto\frac{\alpha^{2}\wedge\omega^{d-2}}{\vol}$$ has signature $(1,h^{1,1}-1)$, which is a purely linear algebraic statement (see Section \ref{sec:linearalgebranotation} for the notation).  This has an easy direct proof since we may diagonalize  $\omega$ with respect to some appropriate basis of $E$ and immediately obtain the representative matrix of $Q_{\omega^{d-2}}$ in terms of such a basis.

In relation with  classical questions arising  in convex geometry,
Alexandrov  considers in  his 1938 paper \cite{alexandrov} the situation in the linear set-up where in the expression of $Q$ the form $\omega^{d-2}$ is replaced by the exterior product $\omega_{1}\wedge\ldots\wedge\omega_{d-2}$ of $d-2$ strictly positive $(1,1)$-forms $\omega_{1},\ldots,\omega_{d-2}$ on $E$. He proves, by a non-trivial algebraic argument, that the signature of the corresponding form $Q_{\omega_{1}\wedge\ldots\wedge\omega_{d-2}}$ is still $(1,h^{1,1}-1)$ in this case. Similarly, in \cite{Gromov90} Gromov proves the ``K\"ahler version'' of this statement, namely that if $\omega_{1},\ldots,\omega_{d-2}$ are K\"ahler forms on the compact complex manifold $X$ then 
$$Q_{\{\omega_{1}\}\ldots\{\omega_{d-2}\}}:H^{1,1}_\R(X)\to\R, \ \{\alpha\}\mapsto\int_X\alpha^2\wedge\omega_{1}\wedge\ldots\wedge\omega_{d-2}$$
has signature $(1,h^{1,1}-1)$. 

In our previous work on this subject, we show in \cite{RossToma} that if $\lambda=(\lambda_{1},\ldots,\lambda_{N})$ is a partition of $d-2$ with $0\le\lambda_{N}\le\ldots\le\lambda_{1}\le e$, the Schur classes $s_{\lambda}(E)$ of ample vector bundles $E$ of rank $e$ on $X$ also have the Hodge-Riemann property. Here the Schur class $s_{\lambda}(E)$ is defined in terms of the Chern classes $c_i:=c_i(E)$ by
$$s_{\lambda}(E) := \det \left(\begin{array}{cccc}c_{\lambda_1} & c_{\lambda_1+1} & \cdots & c_{\lambda_1+N-1}\\
c_{\lambda_2-1} & c_{\lambda_2} & \cdots & c_{\lambda_2+N-2}\\
\vdots & \vdots & \vdots & \vdots  \\
c_{\lambda_N-N+1} & c_{\lambda_N-N+2} & \cdots & c_{\lambda_N}\\
\end{array}\right) \in H^{d-2,d-2}_{\R}(X).$$
Note that the hypothesis that $E$ is ample necessarily implies that $X$ is projective.  Moreover, if $E$ splits as a direct sum of line bundles $E=\oplus_{j=1}^{e} L_{j}$ then $s_{\lambda}(E)$ is a symmetric homogeneous polynomial in the Chern classes $a_{j}:=c_{1}(L_{j})$ of  the line bundles $L_{j}$. We will denote this polynomial by $s_{\lambda}(a_{1},\ldots,a_{e})$.   Our  result is compatible with that of Gromov when taking $e=d-2$ and the partition $\lambda = (d-2)$, for then $s_{\lambda} = c_{d-2}$ and $s_{\lambda}(E) = c_1(L_1) \wedge \cdots \wedge c_{1}(L_{d-2})$.

These considerations allow us to formulate both a linear version and a K\"ahler version of our main result in \cite{RossToma} and  the main purpose of this paper is to prove these  versions.  

\begin{theorem}\label{thm:sketch}
Let $d\ge 2$, $e\ge 1$ be two integers and $\lambda=(\lambda_{1},\ldots,\lambda_{N})$ be a partition of $d-2$ with $0\le\lambda_{N}\le\ldots\le\lambda_{1}\le e$. Let further $E$ be a $d$-dimensional complex vector space and $X$ be a $d$-dimensional compact complex manifold.
The following statements hold.
\begin{enumerate}
\item If $\omega_{1},\ldots,\omega_{e}$ are strictly positive $(1,1)$-forms on $E$ then $s_{\lambda}(\omega_{1},\ldots,\omega_{e})$ has the Hodge-Riemann property, i.e.  the quadratic form 
$$Q_{s_{\lambda}(\omega_{1},\ldots,\omega_{e})}:\bigwedge^{1,1}_\R E^*\to\R, \ \alpha\mapsto\frac{\alpha^{2}\wedge s_{\lambda}(\omega_{1},\ldots,\omega_{e})}{\vol}$$ has signature $(1,h^{1,1}-1)$.
\item If $\omega_{1},\ldots,\omega_{e}$ are K\"ahler forms on $X$ then $s_{\lambda}(\{\omega_{1}\},\ldots,\{\omega_{e}\})$ has the Hodge-Riemann property, i.e.  the quadratic form 
$$Q_{s_{\lambda}(\{\omega_{1}\},\ldots,\{\omega_{e}\})}:H^{1,1}_\R(X)\to\R, \ \{\alpha\}\mapsto\int_X\alpha^2\wedge s_{\lambda}(\omega_{1},\ldots,\omega_{e})$$ has signature $(1,h^{1,1}-1)$.
\end{enumerate}
\end{theorem}

\subsection*{Outline of the proof}

The stages of our proof run as follows.  We first prove (2) in the case that $X$ is a torus with maximal Picard rank.  From this one can deduce the statement (1) rather easily by taking $E$ as the tangent space at a point in such a torus and considering the parallel translates of the positive $(1,1)$ forms  $\omega_j$ on $E$ which become K\"ahler forms on $X$.  From this linear case it is possible to deduce the general case of (2) using a pointwise to global argument.

So the main work is done in proving (1) when $X$ is a torus with maximal Picard rank.    For such an $X$, any K\"ahler form can be perturbed to a rational K\"ahler form.  From our previous work, we know that if each $\omega_j$ is rational then $Q:=Q_{s_{\lambda}(\{\omega_{1}\},\ldots,\{\omega_{e}\})}$ has the  Hodge-Riemann property, and thus we deduce that even without this rationality assumption $Q$ has the {\em weak Hodge-Riemann property}, by which we mean  that it is a limit of intersection forms with the   Hodge-Riemann property.  

Of course that is not enough, so to prove that in fact $Q$ has the Hodge-Riemann property we will consider what happens when replacing each $\omega_i$ with $\omega_i + th$ for some ample class $h$ and small parameter $t$.  This leads us to various families of bilinear forms with the weak Hodge-Riemann property.  
In Section \ref{section:augmentation} we develop a linear algebra machine that considers such families and gives carefully constructed conditions under which the full Hodge-Riemann property can be shown to hold.

 We note that this proof of the linear case still uses geometry, as it relies on our previous work and thus ultimately on the Hard-Lefschetz Theorem.  It would be interesting to know if a purely linear algebra proof is possible.

%\color{blue}
\subsection{Combinations of Schur Classes}
%\marginpar{\small M: tiny changes in this subsection; it is not yet finished}
It is natural to ask if other characteristic classes than the Schur classes given in Theorem \ref{thm:sketch} enjoy the Hodge-Riemann property.   Our linear algebra machine can also be used to give a condition that guarantees that this is the case for certain linear combinations of Schur classes.   To describe this fix $d\ge 2$,  $N\ge d-2$ and $e\ge 1$.  Then for each partition $\lambda$ of $d-2$ we can consider the following two objects
\begin{enumerate}
\item The Schur polynomial $s_{\lambda}(x_1,\ldots,x_e)$,
\item The Schubert class $C_{\lambda}$ inside the Grassmannian $\Grass(N,\mathbb C^{N+e})$ of $N$ dimensional linear subspaces of $\mathbb C^{N+e}$. 
\end{enumerate}
Similarly we can consider linear combinations of these objects.   To this end we will consider sums over all partitions $\lambda = (\lambda_1,\cdots,\lambda_n)$ of $d-2$ such that $e\ge \lambda_1 \ge \lambda_2\ge \cdots \ge \lambda_n >0$.    We order these partitions in some arbitrary but fixed way as $\lambda^{(1)},\ldots,\lambda^{(k)}$ where $k$ is the number of such partitions.

Then for non-negative real numbers $x = (x_1,\ldots,x_k)$ that sum to $1$  consider the polynomial
$$\Gamma_x := \sum_i x_i s_{\lambda^{(i)}}.$$
We say that $\Gamma_x$ has the \emph{universal Hodge-Riemann property} if for all complex manifolds $X$ of dimension $d-2$ and all ample vector bundles $E$ on $X$ of rank $e$ the bilinear form induced by the class $\Gamma_{x}(E)$ on $H^{1,1}_{\mathbb R}(X)$ has the Hodge-Riemann property.  Similarly we say that $\Gamma_x$ has the \emph{Hodge-Riemann property in the linear (resp.\ K\"ahler) case} if the analog of the conclusion of Theorem \ref{thm:sketch}(1) (resp.\ Theorem \ref{thm:sketch}(2)) holds for $\Gamma_x$.

On the other hand for $x$ as above we say that $x$ is {\em irreducibly representable} if all the $x_i$ are rational and there exist positive integers $N\ge d-2$ and $m$ such that the cycle
$$m \sum_i x_i C_{\lambda^{(i)}}$$ 
inside $\Grass(N,\mathbb C^{N+e})$ is algebraically equivalent to an irreducible cycle.  
\begin{theorem}
%We denote by $B$ the set of all irreducibly representable $x$ 

Let $B$ denote the set of all $x$ that are irreducible representable.  Then for all $x\in \overline{B}$ the following hold
\begin{enumerate}
\item The class $\Gamma_x$ has the universal Hodge-Riemann property.
\item The class $\Gamma_x$ has the universal Hodge-Riemann property in the linear case
\item The class $\Gamma_x$ has the universal Hodge-Riemann property in the K\"ahler case.
\end{enumerate}
\end{theorem}

We emphasize that the statement  holds for points in the closure of $B$, 
and passing from $B$ to its closure is non-trivial as the Hodge-Riemann property for bilinear forms is an open but not closed condition. 

\color{blue}\subsection*{Outline of the proof} \color{black} For points inside $B$ we use the irreducible cycle algebraically equivalent to $m \sum_i x_i C_{\lambda^{(i)}}$ to produce a morphism $\pi:\hat{C}\to X$ from an irreducible variety $\hat{C}$ and a nef vector bundle $U$ on $\hat{C}$ so that $\pi_*(c_{n-2}(U)) = \Gamma_x(E)$.   Essentially by our previous work, this is enough to prove the statement in this case.   So by continuity for every $x$ in the closure of $B$ we know that the bilinear form associated to $\Gamma_x$ has the weak Hodge-Riemann property, and we then apply the same linear algebra machinery discussed above.

\subsection*{Comparison with previous work: }  Our first work in this subject is \cite{RossToma} in which we prove the Hodge-Riemann property for Schur classes of ample bundles.   We continue this in \cite{RossToma2} in which we emphasize more the importance of the weak Hodge-Riemann property (which is much easier to prove) and from this we develop various inequalities among characteristic classes of nef vector bundles.  The linear algebra machinery we develop in this paper is an abstraction of the arguments in \cite{RossToma}.  In fact, combining what is written here with \cite{RossToma2} reproves the main results of \cite{RossToma}.

We have already mentioned Gromov's paper \cite{Gromov90} who following work of 
Alexandrov  in the linear case  \cite{alexandrov} initiated the investigation into what classes on K\"ahler manifolds have the Hard-Lefschetz and Hodge-Riemann properties.  This has since been taken up by others, for instance \cite{Cattani,DinhNguyen06,DinhNguyen13}. 

The idea that Schur classes of ample vector bundles have some kind of positivity originates in the work of Fulton-Lazarsfeld \cite{FultonLazarsfeld} who prove that top-degree Schur classes of ample vector bundles are positive.   All our results concerning Hodge-Riemann properties of characteristic classes of ample vector bundles rely on this statement.

\subsection*{Acknowledgements: }  The first author was supported by NSF grants DMS-1707661 and DMS-1749447 during this work.

\section{Preliminaries}

\subsection{Linear Algebra Notation}\label{sec:linearalgebranotation}

Let $E$ be a $d$-dimensional complex vector space.  
We denote by $E^*:=\Hom_\C(E,\C)$, $\bar E^*:=\Hom_{\C-\operatorname{antilin}}(E,\C)$ and by 
$\bigwedge^{p,q}E^*$ the spaces of $(1,0)$, $(0,1)$ and $(p,q)$-forms on $E$, respectively. Recall that $\bigwedge^{p,q}E^*$ is the image of $\bigwedge^pE^*\otimes\bigwedge^q\bar E^*$ in $\bigwedge^{p+q}(E^*\oplus \bar E^{*})$. The conjugation operator $E^* \to\bar E^*$ naturally extends for each bidegree $(p,q)$ as an operator $\bigwedge^{p,q}E^* \to \bigwedge^{q,p}E^*$ and when $p=q$ we denote by $\bigwedge^{p,p}_\R E^*$ its space of fixpoints; it is the space of real $(p,p)$-forms on $E$. Anticipating our application when $E$ will be the holomorphic tangent space at a point of a complex manifold, we denote a basis of $E$ by $(\frac{\partial}{\partial z_{1}},\ldots,\frac{\partial}{\partial z_{d}})$ and its dual basis by $(\d z_{1},\ldots,\d z_{d})$. We have a canonical orientation on $E$ given by the top degree form 
$\vol:=i\d z_{1}\wedge\d\bar z_{1}\wedge\ldots\wedge i\d z_{d}\wedge\d\bar z_{d}$.

\subsection{$\mathbb R$-twisted vector bundles}\label{sec:Rtwistedbundles}Given a vector bundle $E$ on a manifold $X$ and a class $\delta\in H^{1,1}_\mathbb R(X)$ we denote the $\mathbb R$-twisted bundle by $E\langle \delta\rangle$ which is a formal object understood to have Chern classes defined by the rule
\begin{equation}c_p(E\langle \delta \rangle) := \sum_{k=0}^p\binom{e-k}{p-k} c_k(E) \delta^{p-k} \text{ for } 0\le p\le e.\label{eq:defcherntwisted}\end{equation}
Equivalently,  if $x_1,\ldots,x_e$ are the Chern roots of $E$ then $x_1+\delta,\ldots,x_e+\delta$ are by definition the Chern roots of $E\langle \delta\rangle$.  The reader is referred to \cite[Section 6.2, 8.1.A]{Lazbook2} or \cite[Sec. 2.4]{RossToma} for the basic properties of these objects.
\

\subsection{Quadratic Forms and the (weak) Hodge-Riemann property}

Let $V$ be a real vector space of finite dimension $\rho$ and 
$$Q:V\times V\to \mathbb R$$ be a symmetric bilinear form on $V$.    We write
$ Q(v) : = Q(v,v) \text{ for } v\in V$
for the associated quadratic form.

\begin{definitionlemma}[The Hodge-Riemann property]\label{deflemma:HRofquadraticform}
Suppose there exists an $h\in V$ such that $Q(h)>0$.   Then the following statements are equivalent, in which case we say that $Q$ has the \emph{Hodge-Riemann property}.
\begin{enumerate}
\item $Q$ has signature $(1,\rho-1)$.
\item There exists a subspace of dimension $\rho-1$ in $V$ on which $Q$ is negative definite.
\item For any $h'\in V$ such that $Q(h')>0$, the restriction of $Q$ to the primitive space 
$$ V_{h'}:= \{ v\in V : Q(v,h')=0\}$$
 is negative definite.
\item For any $h'\in V$ such that $Q(h')>0$ and all $v\in V$ the \emph{Hodge-Index inequality}
\begin{equation}  Q(v)  Q(h')\le  Q(v,h')^2 \label{eq:hodgeindex}\end{equation}
holds, with equality iff $v$ is proportional to $h'$.
\end{enumerate}
\end{definitionlemma}
\begin{proof}
(1) $\Rightarrow$ (2) and (3) $\Rightarrow$ (1) and (4)$\Rightarrow$ (3) are immediate, and (1)$\Rightarrow$ (3) comes from Sylvester's law of inertia.    For (3)$\Rightarrow$ (4): Given $v\in V$ choose $\lambda$ so  $ Q(v + \lambda h',h')=0$.  By  (3), this implies $ Q(v+\lambda h', v +\lambda h')\le 0$ with equality iff $v + \lambda h'=0$.  Rearranging gives (4). \end{proof}

\begin{definition}\label{def:polarizedHR}
When $Q$ has the Hodge-Riemann property and $h\in V$ is such that $Q(h)>0$ we will say that $Q$ has the \emph{Hodge-Riemann property with respect to $h$}.
\end{definition}

Note that when $Q$ has the Hodge-Riemann property, choosing an element $h$ in the set $\{ v\in V \ | \ Q(v)>0\}$ serves to distinguish one of the two connected components of this set whose elements may be looked upon as ``positive'' vectors for $Q$. %We will denote this component by $C_{Q}^{+}$. We have $C_{Q}^{+}=\{ v\in V \ | \ Q(v)>0, \ Q(v,h)>0\}$. %We also have$C_{Q}^{+}=\{ v\in V \ | \ Q_{V}(v, v')>0 \ \forall v'\in \overline{C_{Q}^{+}}\setminus 0 \}$. 

\begin{definition}\label{def:weakHR}
We will say that $Q$ has the \emph{weak Hodge-Riemann property} if $Q$ is a limit of symmetric bilinear forms on $V$ that have the Hodge-Riemann property. If moreover an element $h\in V$ exists such that $Q(h)>0$, we will say that $Q$ has the \emph{ weak Hodge-Riemann property with respect to $h$}.
\end{definition}

\begin{lemma}\label{lem:wHR}  Let $Q$ be symmetric bilinear form on $V$ as before and suppose that an element $h\in V$ exists such that $Q(h)>0$. Then $Q$ has the weak Hodge-Riemann property with respect to $h$ if and only if 
$$ Q(v)  Q(h)\le  Q(v,h)^2 \ \text{ for all } v\in V.$$
\end{lemma}

\begin{proof} If 
$Q$ satisfies the condition $Q(h)>0$ and is also the limit of a sequence $(Q_n)_{n\in\N}$ of bilinear forms having the Hodge-Riemann property, then we may clearly suppose that the forms $Q_n$ have the Hodge-Riemann property with respect to $h$ and thus they satisfy the Hodge-Index inequality  $ Q_n(v)  Q_n(h)\le  Q_n(v,h)^2$ for all $v\in V$. It follows that $Q$ satisfies the Hodge-Index inequality too. 

Conversely, if $Q$ satisfies the condition $Q(h)>0$ and  the Hodge-Index inequality $ Q(v)  Q(h)\le  Q(v,h)^2$ for all  $v\in V$, then diagonalizing $Q$ with respect to a basis having $h$ as first vector leads to a form where all  diagonal entries but the first one are non-positive. Thus $Q$ has the weak Hodge-Riemann property with respect to $h$. 
 \end{proof}

 \subsection{Previous Results for Vector Bundles}

The proof we give of our main result will depend on our previous work on Schur classes of vector bundles. Here we state and sketch the proofs of two results from \cite{RossToma} and \cite{RossToma2} which will be used in an essential way in  Proposition \ref{prop:Q'nefFnonpositive}. (In fact we will use slight generalizations that allow the base space to be irreducible rather than smooth.)

%Specifically we will use the following special case of \cite[Theorem 7.2]{RossToma2} in an essential way.
 
\begin{theorem}\label{theorem:paper2}\cite[Theorem 7.2]{RossToma2} 
Let $X$ be a complex projective manifold of dimension $d$, let $E$ be a $\mathbb Q$-twisted nef vector bundle on $X$ and let $\lambda$ be a partition of $d-2$.   Then   the quadratic form 
$$Q_{s_{\lambda}(E)}:H^{1,1}_\R(X)\to\R, \ \{\alpha\}\mapsto\int_X\alpha^2\wedge s_{\lambda}(E)$$ has the weak Hodge-Riemann property.
\end{theorem}
\begin{sketchproof}
Rather than repeat the proof here we merely indicate the main ingredients.   The first step is to observe that if $h$ is an ample class on $X$ and $E$ is ample then  $Q_{s_{\lambda}(E)}(h)>0$ by  Fulton-Lazarsfeld \cite{FultonLazarsfeld}.

We start by proving the statement when $E$ has  rank $d-2$ and for the Chern class $c_{d-2}(E)$.  This is done by a continuity argument, replacing $E$ with the $\mathbb R$-twist $E\langle th\rangle$ where $h$ is an ample class on $X$ and $t>0$ is real.        By the Bloch-Gieseker Theorem (which relies on the Hard-Lefschetz Theorem on  the projective bundle $\mathbb P(E)$) we know that $Q_{c_{d-2}(E\langle th \rangle)}$ is always non-degenerate, and so its signature does not change as $t>0$ varies.    But the classical Hodge-Riemann bilinear relations for $h^{d-2}$ imply that this signature is of the expected form for $t$ large, completing the proof in this case.   

The general case then follows from this by a geometric construction that provides a morphism $\pi:\hat{C}_E\to X$ from an irreducible variety $\hat{C}_E$ of dimension $n$ and a nef bundle $U$ on $\hat{C}_E$ of rank $n-2$ so that $\pi_* c_{n-2}(U) = s_{\lambda}(E)$  (see Section \ref{subsec:geometric} for a brief discussion of this construction in the case that $E$ splits as a sum of line bundles).  Then the weak Hodge-Riemann property for $s_{\lambda}(E)$ follows from that for $c_{n-2}(U)$.
\end{sketchproof}\\

Another result that we will need is the following inequality on Chern classes of ample vector bundles.

\begin{theorem}\cite[Theorem 3.2]{RossToma}\cite[Theorem 10.2]{RossToma2}\label{theorem:oldinequality}
Let $X$ be a complex projective manifold of dimension $d$, let $E$ be a  nef vector bundle on $X$ and $h$ be an ample class on $X$.   Then for all $\alpha\in H^{1,1}(X)$ it holds that
\begin{equation}\int_X \alpha^2 c_{d-2}(E) \int_X h c_{d-1}(E) \le 2 \int_X \alpha h c_{d-2}(E) \int_X \alpha c_{d-1}(E).\label{eq:higherchernclassinequality}\end{equation}
\end{theorem}
\begin{sketchproof}
The statement is trivial unless the rank of $e$ is at least $d-1$.   Consider the product $X':= X\times \mathbb P^1$ and the bundle $E': = E\boxtimes \mathcal O_{\mathbb P^1}(1)$.   Then Theorem \ref{theorem:paper2} says that bilinear form $Q_{c_{d-1}(E')}$ has the weak  Hodge-Riemann property on $H^{1,1}_\mathbb R(X') = H^{1,1}_{\mathbb R}(X)\oplus \mathbb R$, so in particular satisfies the Hodge-index inequality.  Applying this to a suitable class in $H^{1,1}_{\mathbb R}(X')$ (see the proof of \cite[Theorem 10.2]{RossToma2}) gives \eqref{eq:higherchernclassinequality}.
\end{sketchproof}

\section{Augmentation}\label{section:augmentation}

Suppose that $\mathcal R_t$ for $t\in \mathbb R$ is a family of bilinear forms on a fixed vector space $V$ each of which has the weak Hodge-Riemann property.  In this section we describe conditions on this family under which it is possible to deduce that $\mathcal R_0$ actually has the Hodge-Riemann property (perhaps after restricting to a certain subspace of $V$).  For lack of a better term we call this process ``augmentation" which is, in the end, nothing but a formal piece of linear algebra.  

To describe this more precisely, let $W$ be a finite dimensional vector space and consider
$$ V: = W\oplus \mathbb R.$$
To ease notation, let $\zeta$ be a vector that spans the extra factor of $\mathbb R$, so each $\beta\in V$ can be written as 
$$\beta = \alpha + \lambda \zeta \text{ for } \alpha\in W \text{ and } \lambda\in \mathbb R.$$
Now suppose $\mathcal R_t$ for $t\in \mathbb R$ is  a family of bilinear forms on  $V$ that we assume is differentiable with respect to $t$, and write $\mathcal R'_t$ for the derivative.  We also fix some non-zero element $h\in W$.

\subsection{Augmentation 1}
\begin{definition}
 We say that $\mathcal R_t$  \emph{has property (A)} if the following hold
\begin{enumerate}[label=(A{\arabic*}), ref=A{\arabic*}]
\item  \label{A1}  $\mathcal R_0(h)>0$ and $\mathcal R'_0(h)>0$.
\item  \label{A2} For $|t|\ll 1$  $\mathcal R_t$ has the weak Hodge-Riemann property with respect to $h$, i.e.
$$\mathcal R_t(\beta) \mathcal R_t(h) \le \mathcal R_t(\beta,h)^2 \text{ for all } \beta\in V. $$
\item \label{A3}  The following inequality holds 
$$ \mathcal R'_0(\beta) \mathcal R_0(h) \le 2 \mathcal R'_0(\beta,h) \mathcal R_0(\beta,h) \text{ for all } \beta\in V.$$ 
\item \label{A4}{} There exists some $c=c_{\cR}\in\R$ 
%\setminus\{0\}$ 
 such that $\mathcal R'_0(\beta,\zeta) = c\mathcal R_0(\beta, h)$ for all $\beta \in V$.
\item \label{A5}{}  $\mathcal R_0(\zeta, h)>0$.
%\item \begin{equation} \mathcal R_t = \mathcal R_0 +  c t  \mathcal Q_0+ O(t^2).\label{eq:Omegaiexpansion}\end{equation}
%for some constant $c>0$.  
\end{enumerate}
\end{definition}

\begin{remark}
The reader should just consider these as abstract properties of the family $\mathcal R_t$.  In our application later  (A2) will follow from Theorem \ref{theorem:paper2} and (A3) will follow from Theorem \ref{theorem:oldinequality}.
\end{remark}

\begin{theorem}\label{thm:augmentation1}
Let $\mathcal R_t$ be family bilinear forms on $V$ that has property (A).    If $\mathcal R'_0$ has the Hodge-Riemann property with respect to $h$ then $\mathcal R_0$ also has the Hodge-Riemann property with respect to $h$.
\end{theorem}

\begin{proof}
Note that we are assuming \eqref{A2} that $\mathcal R_0$ has the weak-Hodge Riemann property, and the task is to show that it actually has the Hodge-Riemann property, and by \eqref{A1} we know $\mathcal R_0(h)>0$.  So suppose $\beta\in V$ is such that
\begin{equation}\mathcal R_0 (\beta,h) = 0 = \mathcal R_0(\beta).\label{eq:assumptiononbeta}\end{equation}
Then the task is to show that $\beta=0$.\\

{\bf Step 1: } We claim that  
\begin{equation} \mathcal R'_0(\beta)=0. \label{eq:Rdotofbeta=0}\end{equation}

To prove this consider
\begin{equation}\label{eq:propBiimpliesAi+1setup3} f(t):=\mathcal R_t(\beta) \mathcal R_t(h) - \left(\mathcal R_t(\beta, h)\right)^2.
\end{equation}
By \eqref{A2}, $f(t)\le 0$ for all $|t|\ll 1$ and \eqref{eq:assumptiononbeta} implies $f(0)=0$. Thus
\begin{align*}
0 = f'(0) = \mathcal R_0(\beta) \mathcal R'_0(h) + \mathcal R_0'(\beta) \mathcal R_0(h) - 2 \mathcal R_0(\beta,h) \mathcal R_0'(\beta, h) = \mathcal R_0'(\beta) \mathcal R_0(h)
\end{align*}
which implies the claim \eqref{eq:Rdotofbeta=0} as $\mathcal R_0(h)>0$ by \eqref{A1}. \\

{\bf Step 2: } Now we use the assumption  that $\mathcal R_0'$ has the Hodge-Riemann property, which in particular implies it is non-degenerate.  Thus there exists a $\gamma\in V$ that is dual to the linear map $\mathcal R_0(\cdot, h)$, i.e. such that
\begin{equation} \mathcal R'_0(\delta,\gamma) = \mathcal R_0(\delta,h) \text{ for all } \delta\in V.\label{eq:dual}\end{equation}
Observe that since $\mathcal R_0(h)>0$ we know that $\gamma\neq 0$.\\

{\bf Step 3: }  We next claim that
\begin{equation}\mathcal R_0'(\beta,h) \mathcal R_0(\gamma,h)=0.\label{eq:claiminstep2} \end{equation}
%\begin{equation}\mathcal R'_0(\beta,\gamma) \mathcal R_0(h) = \mathcal R'_0(\beta, h) \mathcal R_0(\gamma,h)  \text{ for all } \gamma\in V\end{equation}
To see this consider
$$ g(s):= \mathcal R'_0(\beta+s\gamma) \mathcal R_0(h) -2 \mathcal R'_0(\beta+s\gamma,h) \mathcal R_0(\beta+s\gamma,h).$$
By \eqref{A3} we have $g(s)\le 0$ for all $s$,  and by (\ref{eq:assumptiononbeta},\ref{eq:Rdotofbeta=0}) we know $g(0)=0$.  Thus
\begin{align*}
0 = g'(0) &= 2 \mathcal R_0'(\beta,\gamma) \mathcal R_0(h)  - 2\mathcal R_0'(\beta,h) \mathcal R_0(\gamma,h) - 2 \mathcal R_0'(\gamma,h) \mathcal R_0(\beta,h)\\
&=  2 \mathcal R_0(\beta,h) \mathcal R_0(h)  - 2\mathcal R_0'(\beta,h) \mathcal R_0(\gamma,h) \tag{by \eqref{eq:dual} and \eqref{eq:assumptiononbeta}}\\
%& = 2 \mathcal R_0(\beta,h) \mathcal R_0(h)  - 2\mathcal R_0'(\beta,h) \mathcal R_0(\gamma,h) \tag{by \eqref{eq:dual}}\\
& = - 2\mathcal R_0'(\beta,h) \mathcal R_0(\gamma,h) \tag{by \eqref{eq:assumptiononbeta} again}\\
\end{align*}
proving the claim.\\

{\bf Step 4: } We next show that  $\beta$ and $\gamma$ are proportional, i.e. there exists a $\kappa\in \mathbb R$ such that $\beta = \kappa \gamma$.

Suppose first that $\mathcal R'_0(\beta, h)=0$.  Recall we already know from Step 1 that $\mathcal R_0'(\beta)=0$ and by 
\eqref{A1} $\mathcal R'_0(h)>0$.  Thus since $ \mathcal R_0'$ is assumed to have the Hodge-Riemann property we deduce that $\beta=0$ so the Claim certainly holds with $\kappa=0$.

So we may assume $\mathcal R_0'(\beta,h)\neq 0$ and so by \eqref{eq:dual},  $\mathcal R_0'(\gamma) = \mathcal R_0(\gamma,h)=0$ by Step 3.  Thus in summary, the classes $\beta$ and $\gamma$ both lie in $\ker( \mathcal R_0(\cdot,h))$ and also in the null cone of $\mathcal R_0'$.  
Recall $\mathcal R_0'$ has signature $(1,\dim V-1)$ and  by \eqref{A3} is negative semidefinite on $\ker( \mathcal R_0(\cdot, h) )$.  But this is only possible if $\beta$ is proportional to $\gamma$ (this is a formal statement about such bilinear forms that for completeness we include in Lemma \ref{lem:intersectionnullandcodim1:repeat}).  This proves the claim that $\beta= \kappa \gamma$.\\

{\bf Step 5:} We are now ready to complete the proof by observing that
\begin{align*}
\kappa \mathcal R_0(\zeta,h)&=\kappa \mathcal R_0'(\zeta,\gamma) \tag{by definition of $\gamma$ in \eqref{eq:dual}}\\
&= \mathcal R_0'(\zeta,\beta) \tag{as $\beta = \kappa \gamma$}\\
&=c \mathcal R_0(\beta,h) \tag{by \eqref{A4}}\\
&=0 \tag{by our hypothesis on $\beta$ in \eqref{eq:assumptiononbeta}}
\end{align*}
But \eqref{A5} tells us that $\mathcal R_0(\zeta,h)>0$, so we must have $\kappa=0$ and hence $\beta=0$ completing the proof.
\end{proof}

\begin{lemma}\label{lem:intersectionnullandcodim1:repeat}
Let $\mathcal Q$ be a bilinear form on a finite dimensional vector space $V$ with the Hodge-Riemann property.  Let $V'\subset V$ be a subspace on which $\mathcal Q$ is negative semidefinite.    Then if $\beta,\gamma\in V'$ satisfy $\mathcal Q(\beta) = \mathcal Q(\gamma)=0$ and $\gamma\neq 0$ then $\beta= \kappa \gamma$ for some $\kappa\in \mathbb R$.
\end{lemma}
\begin{proof}
Let $h\in V$ be such that $\mathcal Q(h)>0$.  For $t\in \mathbb R$  we have $\beta+t\gamma\in V'$ and hence
$$0 \ge \mathcal Q(\beta + t\gamma) = 2t\mathcal Q(\beta,\gamma).$$
Since this holds for all $t$ we conclude $\mathcal Q(\beta,\gamma)=0$.    Thus we actually have
$$ 0 = \mathcal Q (\beta + t\gamma) \text{ for all } t\in \mathbb R.$$
If $\mathcal Q(\gamma,h)=0$ then as $\mathcal Q(\gamma)=0$ and $\mathcal Q$ has the Hodge-Riemann property we would have $\gamma=0$ which is absurd.  So $\mathcal Q(\gamma,h)\neq 0$.  Thus we may find $t_0$ so $\mathcal Q(\beta + t_0\gamma,h)=0$.  Since also $\mathcal Q(\beta + t_0\gamma) =0$ we deduce from the Hodge-Riemann property of $\mathcal Q$ that $\beta+t_0\gamma=0$ and we are done. 
\end{proof}

\subsection{Recursive version}

\begin{theorem}\label{thm:recursion}
Fix an integer $j\ge 2$ and let $(\mathcal R_{i,t})_{1\le i\le j}$ be a sequence of families of bilinear forms on $V$ with the following properties:
\begin{enumerate}
\item For $i\ge 2$ the families $\mathcal R_{i,t}$ have property $(A)$.
\item For each $i\ge 2$ there exists a positive constant $C_i$ such that \begin{equation}
\mathcal R'_{i,0}=C_i\cR_{i-1,0}.\label{eq:recursion}
\end{equation}
\item $\cR_{1,0}|_W=0.$
\item $\cR_{2,0}|_W$ has the Hodge-Riemann property with respect to $h$.
\item The constant $c_{\cR_2}$ in  condition (\ref{A4}) for $\cR_{2,t}$ is non-zero.
\end{enumerate}
Then the forms $\cR_{i,0}$ have the Hodge-Riemann property with respect to $h$ for all $i\in\{2,\ldots,j\}$.
\end{theorem}
\begin{proof}
We prove the statement by induction on $i$ starting with $i=2$. The inductive step follows from Theorem \ref{thm:augmentation1} and  \eqref{eq:recursion}. So we only have to show that  $\cR_{2,0}$ has the Hodge-Riemann property under our hypotheses.

Assume for contradiction that $\cR_{2,0}$ does not have the Hodge-Riemann property on $V$. Since it does have this property on $W$, there must exist an element $\beta\in V$ of the form $\beta=\alpha+\zeta$ with $\alpha\in W$ such that
\begin{equation}\cR_{2,0} (\beta,h) = 0 = \cR_{2,0}(\beta).\label{eq:assumptionon2beta}\end{equation}
Then we get 
by condition \eqref{A4} for $\cR_{2,t}$ on one hand 
\begin{equation}
\label{eq:A4forR_2}
 \cR_{2,0}'(\beta,\zeta)=c_{\cR_2}\cR_{2,0}(\beta,h)=0
\end{equation}
and 
 by Step 1 of the proof of Theorem \ref{thm:augmentation1} on the other hand 
 \begin{equation}\label{eq:Step1forR_2}
\cR_{2,0}'(\beta)=0.
\end{equation}
Now by \eqref{eq:recursion} we may rewrite \eqref{eq:A4forR_2} and \eqref{eq:Step1forR_2} as a system
\begin{equation}
\left\{\begin{array}{ll}\cR_{1,0}(\beta,\zeta)=0,\\
\cR_{1,0}(\beta)=0\end{array}\right.\end{equation}
which using the vanishing of $\cR_{1,0}$ on $W$ translates into 
\begin{equation}\label{eq:R1}
\left\{\begin{array}{ll}
\cR_{1,0}(\alpha,\zeta)+\cR_{1,0}(\zeta,\zeta)=0\\
2\cR_{1,0}(\alpha,\zeta)+\cR_{1,0}(\zeta,\zeta)=0.
\end{array}\right.
\end{equation}
This implies $\cR_{1,0}(\zeta,\zeta)=0$ which again by \eqref{eq:recursion}  and by condition \eqref{A4} for $\cR_{2,t}$  entails
$$c_{\cR_2}\cR_{2,0}(\zeta,h)=0.$$ 
But since $c_{\cR_2}$ is supposed to be non-zero  this contradicts condition \eqref{A5} for $\cR_{2,t}$.   
\end{proof}

\subsection{Augmentation 2}

It turns out that for the applications we have in mind the above recursive form of the augmentation is not enough, and we will need another form of augmentation  that makes hypotheses on the second derivative but with a weaker conclusion.  We continue with the notation above, so $\mathcal R_t$ is a family of bilinear forms on the vector space $V = W\oplus \mathbb R\zeta$  but we ask it to be twice differentiable in $t$ this time, and $h\in W$ is fixed.

\begin{definition}
 We say that $\mathcal R_t$ is a \emph{has property (B)} if the following hold
\begin{enumerate}[label=(B{\arabic*}), ref=B{\arabic*}]
\item  \label{B1} $\mathcal R_0(h)>0$.
\item  \label{B2} For $|t|\ll 1$  $\mathcal R_t$ has the weak-Hodge Riemann property with respect to $h$,  i.e.
$$\mathcal R_t(\beta) \mathcal R_t(h) \le \mathcal R_t(\beta,h)^2 \text{ for all } \beta\in V. $$
\item \label{B3}  For $|t|\ll 1$ the following inequality holds 
$$ \mathcal R'_t(\beta) \mathcal R_t(h) \le 2 \mathcal R'_t(\beta,h) \mathcal R_t(\beta,h) \text{ for all } \beta\in V.$$ 
\item \label{B4}{}  $\mathcal R''_0(\alpha,\zeta) = 2\mathcal R_0'(\alpha, h)$ for all $\alpha \in W$.
\item \label{B5}{}   $\mathcal R''_0(\zeta,\zeta) = 2\mathcal R_0(h)$.
%\item \begin{equation} \mathcal R_t = \mathcal R_0 +  c t  \mathcal Q_0+ O(t^2).\label{eq:Omegaiexpansion}\end{equation}
%for some constant $c>0$
\end{enumerate}
\end{definition}

\begin{theorem}\label{thm:augmentation2}
Let $\mathcal R_t$ be family bilinear forms on $V$ that has property (B).    If $\mathcal R''_0$ has the Hodge-Riemann property with respect to $h$ then the restriction $\mathcal R_0|_W$ has the Hodge-Riemann property with respect to $h$.
\end{theorem}

\begin{proof}
We are assuming \eqref{B2} that $\mathcal R_0$ has the weak-Hodge Riemann property, and since $h\in W$ \eqref{B1} implies that $\mathcal R_0|_W$ also has the weak Hodge-Riemann property.  Thus the task is to show that it actually has the Hodge-Riemann property, and by \eqref{B1} we know $\mathcal R_0(h)>0$.  So suppose $\alpha\in W$ is such that
\begin{equation}\mathcal R_0 (\alpha,h) = 0 = \mathcal R_0(\alpha).\label{eq:assumptiononbetaII}\end{equation}
Then task is to show that $\alpha=0$.\\

{\bf Step 1: } We claim that 
\begin{equation} \mathcal R'_0(\alpha)=0. \label{eq:Rdotofbeta=0II}\end{equation}
which is proved exactly as in Step 1 of the proof of Theorem \ref{thm:augmentation1} and is omitted.\\
%To prove this consider
%\begin{equation}\label{eq:propBiimpliesAi+1setup3II} f(t):=\mathcal R_t(\alpha) \mathcal R_t(h) - \left(\mathcal R_t(\alpha, h)\right)^2.
%\end{equation}
%By \eqref{B2}, $f(t)\le 0$ for all $|t|\ll 1$ and \eqref{eq:assumptiononbeta} implies $f(0)=0$. Thus
%\begin{align*}
%0 = f'(0) = \mathcal R_0(\alpha) \mathcal R'_0(h) + \mathcal R_0'(\alpha) \mathcal R_0(h) - 2 \mathcal R_0(\alpha,h) \mathcal R_0'(\alpha, h) = \mathcal R_0'(\alpha) \mathcal R_0(h)
%\end{align*}
%which implies the claim \eqref{eq:Rdotofbeta=0II} as $\mathcal R_0(h)>0$ by \eqref{B1}. \\

{\bf Step 2: } Consider now

$$g(t):= \mathcal R_t'(\alpha)\mathcal R_t(h) - 2 \mathcal R_t'(\alpha,h) \mathcal R_t(\alpha,h).$$
Then by (\ref{eq:assumptiononbetaII},\ref{eq:Rdotofbeta=0II}) we have $g(0)=0$ and \eqref{B3} implies $g(t)\le 0$ for $|t|\ll 1$. Thus
\begin{align*} 0 
&= g'(0) = \mathcal R_t''(\alpha) \mathcal R_t(h) + \mathcal R_t'(\alpha) \mathcal R_t'(h) - 2\mathcal R_t''(\alpha,h) \mathcal R_t(\alpha,h) - 2 \mathcal R'_t(\alpha,h)^2|_{t=0}\\
&= \mathcal R_{0}''(\alpha) \mathcal R_0(h) - 2 \mathcal R_0'(\alpha,h)^2\\
&=\frac{1}{2}\left( \mathcal R_0''(\alpha) \mathcal R_0''(\zeta) - \mathcal R_0''(\alpha,\zeta)^2\right)\tag{by (\ref{B4},\ref{B5})}
\end{align*}
Thus 
$$\mathcal R_0''(\alpha) \mathcal R_0''(\zeta) =\mathcal R_0''(\alpha,\zeta)^2.$$
Now we are assuming that $\mathcal R_0''$ has the Hodge-Riemann property with respect to $h$, and since $\mathcal R_0''(\zeta) = 2\mathcal R_0(h)>0$ the Hodge-index inequality for $\mathcal R_0''$ also holds with respect to $\zeta$ (see Definition-Lemma \ref{deflemma:HRofquadraticform}(4)).  So we deduce that $\alpha$ is proportional to $\zeta$,  but since $\alpha\in W$ this is only possible if $\alpha=0$ and we are done. 
\end{proof}

\begin{remark}
 The following example shows that a naive approach to get an Augmentation 1 type result like Theorem \ref{thm:augmentation1} is bound to fail (that is, some additional hypothesis like those appearing in condition (A) are really needed).   Let $\cR_{t}$ be a family of quadratic forms on a vector space $V$ of dimension $n\ge1$ as above and suppose that there exists a non-zero element $h\in V$ such that:
\begin{enumerate}
\item $\cR_{t}$ has the weak Hodge-Riemann property with respect to $h$ for $|t|\ll 1$,
\item $\cR_{t}$ has the  Hodge-Riemann property for $|t|\ll 1$ and $t\ne0$,
\item $\cR'_{t}$ has the  Hodge-Riemann property with respect to $h$ for $|t|\ll 1$.
\end{enumerate}
Then either $\mathcal R_0$ has the Hodge-Riemann property or the kernel $N$ of $\cR_{0}$ is one-dimensional.  Moreover the latter case can happen as soon as $n>1$.
\end{remark}

\begin{proof} By assumption $\dim N\ge 1$. Let $\beta$ be any vector in $N$. We consider
\begin{equation} f(t):=\mathcal R_t(\beta) \mathcal R_t(h) - \left(\mathcal R_t(\beta, h)\right)^2.
\end{equation}
 Then $f(t)\le 0$ for all $|t|\ll 1$ and $f(0)=0$. Thus
\begin{align*}
0 = f'(0) = \mathcal R_0(\beta) \mathcal R'_0(h) + \mathcal R_0'(\beta) \mathcal R_0(h) - 2 \mathcal R_0(\beta,h) \mathcal R_0'(\beta, h) = \mathcal R_0'(\beta) \mathcal R_0(h)
\end{align*} 
and so $\beta$ must lie in the null cone $C_{\cR'_{0}}^{0}:=\{\alpha\in V \ | \ \mathcal R_0'(\alpha)=0\}$ of $\mathcal R_0'$. But by Lemma \ref{lem:intersectionnullandcodim1:repeat} $C_{\cR'_{0}}^{0}$ cannot contain two-dimensional vector subspaces since $\cR'_{0}$ has the Hodge-Riemann property and the first assertion follows. 

For the second assertion it is enough to consider the following  family of quadratic forms on $\R^{n}$:
$$\cR_{t}(x)=(1+t)x_{1}^{2}+2x_{1}x_{2}+(1-t)x_{2}^{2}+\sum_{i=3}^{n} (-1-t)x_{i}^{2}.$$
\end{proof}

\section{Schur polynomials of K\"ahler forms on tori with maximal Picard number and the linear case of Theorem \ref{thm:sketch}}\label{sec:cone_classes}

The main result of this section is the following Theorem \ref{thm:main} which is  Theorem \ref{thm:sketch}(2) in the special case when $X$ is a torus with maximal Picard number, i.e. $X$ is a torus $T$ such that $N^1_{\mathbb R}(T)=H^{1,1}_\R(T)$. From it we will easily be able to deduce in Subsection \ref{subsec:proofSketch}  the linear case of our main result , i.e. Theorem \ref{thm:sketch}(1). Note that tori with maximal Picard number exist in every dimension, e.g $(\C/\Z[i])^d$, see \cite{Beauville2014} for results on projective manifolds with maximal Picard number.

\begin{theorem}\label{thm:main}
Let $d\ge 2$, $e\ge 1$ be two integers and $\lambda=(\lambda_{1},\ldots,\lambda_{N})$ be a partition of $d-2$ with $0\le\lambda_{N}\le\ldots\le\lambda_{1}\le e$. Let further $T$ be a $d$-dimensional  complex torus with maximal Picard number and let  $\omega_{1},\ldots,\omega_{e}$ be K\"ahler forms on $T$. Then $s_{\lambda}(\{\omega_{1}\},\ldots,\{\omega_{e}\})$ has the Hodge-Riemann property. 
%, i.e.  the quadratic form $$Q_{s_{\lambda}(\{\omega_{1}\},\ldots,\{\omega_{e}\})}:H^{1,1}_\R(X)\to\R, \ \{\alpha\}\mapsto\int_X\alpha^2\wedge s_{\lambda}(\omega_{1},\ldots,\omega_{e})$$ has signature $(1,h^{1,1}-1)$.
\end{theorem}

Since the statement is certainly satisfied when $2\le d\le3$ or when $e=1$, we will suppose for the rest of this section that $d\ge4$ and $e\ge2$.\\

%Let $\underline{\omega}:=(\omega_1,\ldots,\omega_e)$ be a vector of K\"ahler classes on a compact $X$ of dimension $d$.  I assume that these all lie in $N^1_{\mathbb R}(X)$.  I think this is fine for a torus, and we know how to reduce to this case.  

We will work in the following set-up: 
\begin{itemize}
\item $X$ is a $d$-dimensional  compact complex manifold (which we will later take to be the torus $T$).
\item $e$ is a positive integer and $\lambda=(\lambda_{1},\ldots,\lambda_{N})$ a partition of $d-2$ with $0\le\lambda_{N}\le\ldots\le\lambda_{1}\le e$.
\item $W:=H^{1,1}_\R(X)$.
\item $h$ is a fixed ample class on $X$.
\item $\hat{X} : = X\times \mathbb P^d$ and  $\zeta$ will denote the hyperplane class on $\mathbb P^d$.  
\item  $V:= H^{1,1}_\R(\hat{X}) = W \oplus \mathbb R{\zeta}$.
\item For a vector $\underline{a}:=(a_1,\ldots,a_e)\in V^{\oplus e}$
we denote by   $s_{\lambda}(\underline{a})$ its Schur polynomial $s_{\lambda}(a_{1},\ldots,a_{e})$ as defined in the Introduction and by $s_{\lambda}^{(j)}(\underline{a})$ its derived Schur polynomials that are defined by requiring 
$$s_{\lambda}(a_1+x,\ldots,a_e+x) = \sum_{j=0}^{|\lambda|} s_{\lambda}^{(j)}(\underline{a}) x^j\text{ for all }x\in V$$ 
 (this notation and terminology originates in \cite{RossToma}). 
Thus $$s_{\lambda}^{(j)}(\underline{a})\in H^{|\lambda|-j,|\lambda|-j}_\R(\hat{X}) \text{ for }0\le j \le|\lambda|,$$ and for $j$ outside this range we set $s_{\lambda}^{(j)}(\underline{a})$ to be zero. 
 \item $\underline{\omega}:=(\omega_1,\ldots,\omega_e)\in W^{\oplus e}$ is a vector of K\"ahler classes on $X$.
\end{itemize}

\begin{definition}\label{def:definitionQ_i} %Let $0\le i\le d$. 
  Define $\hat{\omega}_j: = \omega_j + \zeta$ for $1\le j\le e$ and $\hat{\underline{\omega}}: = (\hat{\omega}_1,\ldots,\hat{\omega}_e)$.
Define a symmetric bilinear form $\mathcal Q_i(\cdot,\cdot;\underline{\omega})$ form on $V$  by
\begin{align*}
\mathcal Q_i(\beta,\beta'):=\mathcal Q_i(\beta,\beta';\underline{\omega}):= \int_{\hat{X}} \beta s_{\lambda}(\underline{\hat{\omega}})\zeta^i h^{d-i} \beta'\end{align*}
for $\beta,\beta'\in V$. 
This makes sense for $0\le i\le d$. For $i$ outside this range we set $\cQ_i=0$.
\end{definition}

\begin{remark}\label{rem:alternativeDef}
This definition is made clearer if one considers the expansion $s_{\lambda}(\underline{\hat{\omega}}) = \sum_j s_{\lambda}^{(j)}(\underline{\omega})\zeta^j$ which reveals
\begin{equation}\label{eq:alternativedefinitionQ_i}
\mathcal Q_i(\beta,\beta') =\left\{\begin{array}{ll}%(d-i)! 
\int_X \beta s_{\lambda}^{(d-i)}(\underline{\omega}) h^{d-i} \beta' & \text{ if }\beta,\beta'\in W \\ 
%(d-i-1)!
\int_X \beta s_{\lambda}^{(d-i-1)}(\underline{\omega}) h^{d-i} & \text{ if }\beta \in W \text{ and } \beta'=\zeta\\ 
%(d-i-2)!
\int_X s_{\lambda}^{(d-i-2)}(\underline{\omega}) h^{d-i} & \text{ if }\beta=\beta'=\zeta
 \end{array}\right.\end{equation}
In fact one can equally take the definition of $\mathcal Q_i$ to be \eqref{eq:alternativedefinitionQ_i} and then require symmetry and extend linearly in each variable.  Note that with our conventions the three formulas hold for all $i$ but they may give non-zero values only when $2\le i\le d$ for the first one, when $1\le i\le d-1$ for the second one and when $0\le i\le d-2$ for the third one.  
\end{remark}

By definition \begin{equation}\label{eq:conditionINDforQ_i}
\mathcal Q_i(\beta, \zeta) = \mathcal Q_{i+1}(\beta, h) \text{ for } \beta\in V %i=0,\ldots,d-1 
\text{ and all } i,\end{equation}
and from this it is clear that for $\alpha\in W$, $\lambda\in \mathbb R$ and all $i$ 
\begin{align}
\mathcal Q_{i}(\alpha + \lambda \zeta, h)& = \mathcal Q_{i}(\alpha,h) + \lambda \mathcal Q_{i+1}(h), \label{eq:linearidentity}\\
\mathcal Q_{i}(\alpha + \lambda \zeta) &= \mathcal Q_i (\alpha) + 2\lambda \mathcal Q_{i+1}(\alpha,h) + \lambda^2\mathcal Q_{i+2}(h).\label{eq:squareidentity}
\end{align}

\begin{definition}\label{def:definitionR_i} Under the above notations we set for $t\in\R$ and  for $0\le i\le d$ 
$$\cR_{i,t}:=\sum_{k=0}^i \binom{d-i+k}{k}t^k\cQ_{i-k}.$$
 For $i$ outside the given range we set $\cR_{i,t}=0$.
\end{definition}

 In particular we have
 $$ \cR_{i,t}=\cQ_i+(d-i+1)t\cQ_{i-1}+O(t^2)$$
for $0\le i\le d$. Moreover a direct computation shows that 
\begin{equation}
\label{eq:R'}
\cR'_{i,t}=(d-i+1)\mathcal R_{i-1,t}\end{equation}
for all $i$ and all $t$.

Our aim is to apply the augmentation results from the previous section to the families $\mathcal R_{i,t}$, so we will need to check that the various hypotheses hold true.  For example we will need to show that each $\mathcal R_{i,t}$ has the weak Hodge-Riemann property.   Note that this is a closed condition, so as long as we assume that we can perturb each $\omega_j$ to a rational K\"ahler class then we may assume this rationality, and then this weak Hodge-Riemann property will follow by realizing $\cR_{i,t}$ geometrically as we do in the next section.

\subsection{Geometric realization of the forms $\cR_{i,t}$}
\label{subsec:geometric}
In this subsection we show how one may realize  the forms $\mathcal R_i$ in the case that each $\omega_j$ is rational using a classical construction which relates Schur classes and degeneracy loci of morphisms of vector bundles.  

The rationality assumption we need is that the classes $\omega_j$ are the (real) Chern classes of line bundles $L_j$ on $X$. That is, we suppose that $\omega_j=c_1(L_j)\in  H^{1,1}_\R(X)$ 
 for $1\le j\le e$ and consider the vector bundle $$E:= \bigoplus_{j=1}^e L_j.$$ By possibly adding zeroes to the partition $\lambda$ we may take its length  $N$ to be arbitrarily large. We will suppose here that $N\ge d-2$. Set $a_i: = e + i -\lambda_i$ for $1\le i\le N$, fix a vector space $H$ of dimension $e+N$ and a nested sequence of subspaces $0\subsetneq A_1\subsetneq A_{2} \subsetneq\cdots \subsetneq A_N\subset H$ with $\dim(A_i) = a_i$.
Next set $F: = H^*\otimes E=\Hom(H,E)$, let $f + 1 = \rank(F) = e(e+N)$ and define
$$\tilde{C}_E : = \{ \sigma\in \Hom(H,E): \dim \ker(\sigma(x))\cap A_i \ge i \text{ for all } i=1,\ldots, N \text{ and } x\in X\}$$
which is a cone in $F$.
Set further $\hat C_E: = [\tilde{C}_E] \subset \mathbb P_{\sub}(F).$ Finally we denote by $\pi:\P_{\sub}(F)\to X$ the projection morphism and by $U$ the  quotient bundle of $\pi^* F $ on $\P_{\sub}(F)$ which sits in the tautological exact sequence
\begin{equation}0 \to \mathcal O_{\mathbb P_{\sub}(F)}(-1) \to \pi^* F \to U \to 0.\label{eq:tautologicalsequence}\end{equation} 
We note that $$H^{1,1}_\R(\mathbb P_{\sub}(F))=\pi^*H^{1,1}_\R(X)\oplus\R\xi=\pi^*W\oplus \R\xi,$$ where $\xi:=c_1(O_{\mathbb P_{\sub}(F)}(1))$. 
We also note that $\hat C_E$ is locally a product with irreducible fibers over $X$ and that $\dim(\hat C_E)=f+2$ (the notation thus far follows that of \cite[Section 5]{RossToma}).

\begin{proposition}\label{prop:pushforwardschur} Let $0\le i\le d$. Then under the isomorphism $V=W\oplus \R\zeta\to \pi^*W\oplus\R\xi=H^{1,1}_\R(\mathbb P_{\sub}(F))$ acting as $\pi^*$ on $W$ and mapping $\zeta$ onto $\xi$, the bilinear form $\cR_{i,t}$ on $V$ corresponds to the form
$$\cS_{i,t}(\beta,\beta'):=\int_{\hat C_E}\beta c_{f-(d-i)}(U\langle \pi^*(th)\rangle)\pi^*(h)^{d-i}\beta'$$
on $H^{1,1}_\R(\mathbb P_{\sub}(F))$.   Here we use the notation $U\langle \delta\rangle$ for an $\R$-twisted vector bundle as described in Section \ref{sec:Rtwistedbundles}.
\end{proposition}
\begin{proof} 
Note first that $U$ has rank $f$. So for any $\delta\in H^{1,1}_\R(\mathbb P_{\sub}(F))$
\begin{equation}c_p(U\langle \delta \rangle) := \sum_{k=0}^p\binom{f-k}{p-k} c_k(U) \delta^{p-k} \text{ for } 0\le p\le f.\label{eq:defcherntwisted}\end{equation}
Replacing $k$ by $p-j$ and putting $p=f-(d-i)$ and $\delta= \pi^*(th)$ gives
 \begin{equation}c_{f-(d-i)}(U\langle\pi^*(th) \rangle) := \sum_{j=0}^{f-(d-i)}\binom{(d-i)+j}{j} t^j c_{f-(d-i)-j}(U) \pi^*(h)^{j}.\label{eq:specialCherntwisted}\end{equation}
 Plugging this into the expression of $\cS_{i,t}$ we obtain
\begin{align*}
\cS_{i,t}(\beta,\beta')&=\int_{\hat C_E}\beta \beta'\sum_{j=0}^{f-(d-i)}\binom{(d-i)+j}{j} t^j c_{f-(d-i)-j}(U) \pi^*(h)^{j+d-i}\\
&=\int_{\hat C_E}\beta \beta'\sum_{j=0}^{i}\binom{(d-i)+j}{j} t^j c_{f-(d-i)-j}(U) \pi^*(h)^{j+d-i},
\end{align*}
where we have used that our choice of $N$ implies $f-d\ge0$, hence $f-(d-i)\ge i$, and also that $h ^{j+d-i}=0$ for $j>i$ for dimension reasons.

By comparing to 
 the definition of $\cR_{i,t}$ we thus see that in order to prove the Proposition it suffices to check that for $0\le l\le d$ the bilinear forms $\cQ_l$ on $V$ correspond to the forms
  \begin{equation}\label{eq:eqv_form}
  (\beta,\beta')\mapsto   \int_{\hat C_E}\beta \beta' c_{f-(d-l)}(U) \pi^*(h)^{d-l} \end{equation}
 on $H^{1,1}_\R(\mathbb P_{\sub}(F))$ .
To check this we will use the identity
$$\pi_*(c_{f-(d-j)}(U)|_{\hat C_E})=s_\lambda^{(d-j)}(E)$$
which holds for $2\le j\le d$ by \cite[Proposition 5.2]{RossToma} as well as the following formulas
$$c_p(U)\xi\pi^*(\eta)= c_{p+1}(U)\pi^*(\eta),
$$
$$c_p(U)\xi^2\pi^*(\eta)= c_{p+2}(U)\pi^*(\eta)
$$
for $\eta\in H^{k,k}(X)_\R$ and $p+k\ge d$, which are easily deduced from the exact sequence \eqref{eq:tautologicalsequence}
\cite[Lemma 4.17]{RossToma}. With these at hand we rewrite the form  \eqref{eq:eqv_form}
in the three cases of equation \eqref{eq:alternativedefinitionQ_i} and get for $0\le l\le d$: 
\begin{itemize}
\item 
if $\beta,\beta'\in W$
$$
  (\pi^*(\beta),\pi^*(\beta'))\mapsto   \int_{\hat C_E}\pi^*(\beta \beta'h^{d-l} ) c_{f-(d-l)}(U)= $$ $$\int_X \beta \beta'h^{d-l}\pi_*(c_{f-(d-l)}(U)|_{\hat C_E})=
  \int_X \beta s_\lambda^{(d-l)}(E) h^{d-l}\beta'= \cQ_l(\beta,\beta'),$$
  \item
   if $\beta\in W$ and $\beta'=\xi$
$$
  (\pi^*(\beta),\xi)\mapsto   \int_{\hat C_E}\pi^*(\beta h^{d-l} ) c_{f-(d-l)}(U)\xi=  \int_{\hat C_E}\pi^*(\beta h^{d-l} ) c_{f-(d-l-1)}(U)=$$ $$\int_X \beta h^{d-l}\pi_*(c_{f-(d-l-1)}(U)|_{\hat C_E})=
  \int_X \beta s_\lambda^{(d-l-1)}(E) h^{d-l}= \cQ_l(\beta,\zeta),$$
  \item
  if $\beta=\beta'=\xi$
$$
  (\xi,\xi)\mapsto   \int_{\hat C_E}\pi^*(h^{d-l} ) c_{f-(d-l)}(U)\xi^2=  \int_{\hat C_E}\pi^*( h^{d-l} ) c_{f-(d-l-2)}(U)=$$ $$\int_X  h^{d-l}\pi_*(c_{f-(d-l-2)}(U)|_{\hat C_E})=
  \int_X  s_\lambda^{(d-l-2)}(E) h^{d-l}= \cQ_l(\zeta,\zeta),$$
\end{itemize}
which proves the Proposition.
\end{proof}

\subsection{Basic Properties of $\mathcal R_i$}
\label{subsec:propertiesR_i}
We are now ready to show that, as long as each $\omega_j$ can be approximated by rational K\"ahler clases, the $\mathcal R_{i,t}$ satisfy all the conditions in property (A) and that $\mathcal R_{d,t}$ satisfies the conditions in property (B).    As the reader will see, most of the properties we prove are formal from the definition of $\mathcal R_i$, with the exception of (A2,A3,B2,B3) that will use the geometric realization from the previous section.

\begin{lemma}[A1,B1]\label{lem:A1forQ}
It holds that
\begin{equation}\cR_{i,0}(h)=\mathcal Q_i(h) > 0 \text{ for } 2\le i\le d\label{eq:thmfl}\end{equation}
and
\begin{equation}\cR'_{i,0}(h)=(d-i+1)\mathcal Q_{i-1}(h) > 0 \text{ for } 3\le i\le d.\end{equation}
\end{lemma}
\begin{proof}
The first follows as $s_{\lambda}^{(i)}$ is monomial positive so $\mathcal Q_i(h)$ is a sum of positive coefficients times integrals of strictly positive classes, and so is strictly positive (the reader may observe that this is the analog of Fulton-Lazarsfeld positivity for ample vector bundles, but is trivial in the case we are considering).  The second statement is immediate from the definition.
\end{proof}

\begin{proposition}[A2,A3,B2,B3]\label{prop:Q'nefFnonpositive}
 Assume that each $\omega_j$ is the limit of rational K\"ahler classes.
\begin{enumerate}
\item For all $2\le i\le d$ and for $|t|\ll 1$  the quadratic forms $\mathcal R_{i,t}$ have the weak Hodge-Riemann property with respect to $h$.
\item For all $i$,
  for all $\beta\in V$  and  for $|t|\ll 1$ we have
\begin{equation}
\mathcal R'_{i,t}(\beta) \mathcal R_{i,t}(h) \le 2 \mathcal R'_{i,t}(\beta,h) \mathcal R_{i,t}(\beta,h).
\label{eq:B3proof}
\end{equation} 
\end{enumerate}
\end{proposition}

%\begin{proof}
%Proof 1: By Proposition \ref{prop:agreeinrationalcase}, when $\omega_i$ are both integral this follows from our first paper  (and I observe is a consequence ultimately of the Boch-Gieseker Theorem for the nef universal bundle on a projective bundle).     Thus by rescaling it holds whenever all the $\omega_i$ are rational.  And since we are assuming that $\omega_i\in N^1_{\mathbb R}(X)$ we can perturb to the rational case and take the limit.
%\end{proof}

\begin{proof}
The statement is closed under variation of $\omega_j$ so by our hypothesis is is enough to prove it when the classes $\omega_j$ are rational.   Next by multiplying all $\omega_j$  by an appropriate positive integer, we may also assume that they are all integral.
%,h,\zeta$.  
%Thus we may replaced $\zeta$ and $h$ with ample rational classes $\zeta',h'$, and then by rescaling some positive multiple of $\zeta'^i h^{'d-i}$ is represented by a smooth $Y\subset X\times \mathbb P^d$ of dimension $d$.    Next, since $X$ is assumed to be very algebraic, each $\omega_j$ can be approximated by a rational class, and so it is sufficient to assume all the $\omega_j$ are rational, and then by multiplying them all by a positive integer, that they are all integral.  
Then $\omega_j= c_1(L_j)$ for some ample line bundles $L_j$ on $X$ and we set
$$ E = \bigoplus_{j=1}^e L_j$$
which is ample. By construction we now have
$s_{\lambda}(\underline{\omega}) = s_{\lambda}(E)$ and by Proposition \ref{prop:pushforwardschur} we reduce ourselves to checking the corresponding properties for the quadratic forms $\cS_{i,t}$ on 
$H^{1,1}_\R(\mathbb P_{\sub}(F))$. 

We first check that the forms $\cS_{i,t}$
 have the weak Hodge-Riemann property for $2\le i\le d$ and for $|t|\ll 1$.  Indeed for $\beta,\beta'\in H^{1,1}_\R(\mathbb P_{\sub}(F))$ 
 we have $\cS_{i,t}(\beta,\beta'):=\int_{\hat C_E}\beta c_{f-(d-i)}(U\langle \pi^*(th)\rangle)\pi^*(h)^{d-i}\beta'$
and we know that   $U\langle \pi^*(th)\rangle$ is nef on $\mathbb P_{\sub}(F)$ for $|t|\ll 1$ as a formal quotient of the nef $\R$-twisted vector bundle $\pi^*(E\langle th\rangle)$. 

Now we are in a position to inject the geometric argument from our previous work, and conclude that  the weak Hodge-Riemann property holds for $\cS_{i,t}$ by 
\cite[Lemmata 6.6 and 7.1]{RossToma2}. (This is essentially an application of 
 Theorem \ref{theorem:paper2}, but since $\hat{C}_E$ is irreducible but not necessarily smooth one must pass to a resolution of singularities  {\em loc. cit.}.)

For the second statement we rewrite the inequality \eqref{eq:B3proof} using \eqref{eq:R'} as
\begin{equation}
\mathcal R_{i-1,t}(\beta) \mathcal R_{i,t}(h) \le 2 \mathcal R_{i-1,t}(\beta,h) \mathcal R_{i,t}(\beta,h).
\label{eq:B3rewritten}
\end{equation} 
It is enough now to prove \eqref{eq:B3rewritten} for $1\le i\le d$, since for the remaining  $i$ it is trivially verified. But  for $1\le i\le d$ we rephrase it by our geometric interpretation as 
\begin{equation}
\mathcal S_{i-1,t}(\beta) \mathcal S_{i,t}(h) \le 2 \mathcal S_{i-1,t}(\beta,h) \mathcal S_{i,t}(\beta,h),
\label{eq:B3rerewritten}
\end{equation} 
which holds by \cite[Corollary 3.4]{RossToma} for $|t|\ll 1$. (This is essentially Theorem \ref{theorem:oldinequality} only once again we have to allow for the fact that $\hat{C}_E$ is not smooth.) Here we have used again that the  $\R$-twisted vector bundle
$U\langle \pi^*(th)\rangle$ is nef on $\mathbb P_{\sub}(F)$ for $|t|\ll 1$.
\end{proof}

\begin{lemma}[A4]\label{lem:A4forQ}
$$\mathcal R_{i,0}'(\beta,\zeta) = (d-i+1) \mathcal R_{i,0}(\beta,h) \text{ for all } \beta\in V \text{and all } i$$
\end{lemma}
\begin{proof}
This follows immediately from \eqref{eq:conditionINDforQ_i}.
\end{proof}

\begin{lemma}[A5]\label{lem:A5forQ}
$$\mathcal R_{i,0}(\zeta,h) >0 \text{ for } 1\le i\le d-1.$$
\end{lemma}
\begin{proof}
$$\mathcal R_{i,0}(\zeta,h)  = \mathcal Q_i(\zeta,h) = \mathcal Q_{i+1}(h,h)>0\text{ for } 1\le i\le d-1$$
by \eqref{eq:conditionINDforQ_i} and Lemma \ref{lem:A1forQ}.
\end{proof}

\begin{remark}
The reader may find it useful to observe that condition (A5) fails when $i=d$, and thus the need for the second augmentation result.
\end{remark}

\begin{lemma}[B4,B5]
$$\mathcal R_{d,0}''(\beta,\zeta) = 2 \mathcal R_{d,0}'(\beta,h)\text{ for all } \beta\in V \text{ and}  $$
$$\mathcal R_{d,0}''(\zeta,\zeta) = 2\mathcal R_{d,0}(h).$$
\end{lemma}
\begin{proof}
These follow straight from the definition.  Actually by differentiating in \eqref{eq:R'} we get
\begin{equation}
\cR''_{i,t}=(d-i+2)(d-i+1)\mathcal R_{i-2,t}%=(d-i+2)\mathcal R_{i-1,t},
\end{equation}
which together with \eqref{eq:conditionINDforQ_i}  gives
$$\cR''_{d,0}(\beta, \zeta)=2\mathcal R_{d-2,0} (\beta, \zeta)= 2\mathcal Q_{d-2} (\beta, \zeta)=2\mathcal Q_{d-1} (\beta, h)=2\mathcal R'_{d-1} (\beta, h)$$
for all $\beta\in V$ and in the particular case when $\beta=\zeta$
$$\cR''_{d,0}(\zeta, \zeta)=2\mathcal Q_{d-1} (\zeta, h)=2\mathcal Q_{d} (h, h)=2\mathcal R_{d} ( h).$$
%Both sides of the first equation are equal to $2\int_X \alpha s_{\lambda}^{(1)}(\underline{\omega}) h^2$ and both sides of the second equation are equal to $2\int_X s_{\lambda}^{(0)}(\underline{\omega}) h^2$
\end{proof}

With this we are ready to invoke our recursive argument.  Before doing this we check two more properties.

\begin{lemma}
$\mathcal R_{1,0}|_W=0$ and $\mathcal R_{2,0}|_W$ has the Hodge-Riemann property.
\end{lemma}
\begin{proof}
We have $\mathcal R_{1,0}=\cQ_1$ and we have already seen that $\cQ_1|_W$ vanishes,  Remark \ref{rem:alternativeDef}.

Now for $\alpha,\alpha'\in W$ %Then$$\mathcal R_{1,0}(\alpha,\alpha') = \mathcal Q_1(\alpha,\alpha') = \int_{X} \alpha s_{\lambda}^{(d-1)}(\underline{\omega}) h^{d-1}=0$$ where the last equality comes from $\dim X=d$.
$$\mathcal R_{2,0}(\alpha,\alpha') = \mathcal Q_2(\alpha,\alpha') = \int_{X} \alpha s_{\lambda}^{(d-2)}(\underline{\omega})\alpha' h^{d-2}$$
and $s_{\lambda}^{(d-2)}(\underline{\omega})\in \mathbb R_{>0}$ is a positive constant.  So, up to this positive constant,  $\mathcal R_{2,0}|_W$ is the intersection form on $H^{1,1}_\mathbb R(X)$ given by intersecting with $h^{d-2}$ which has the Hodge-Riemann property by the classical Hodge-Riemann bilinear relations as $h$ is assumed to be ample.
\end{proof}

\subsection{Proof of Theorem \ref{thm:main}}\label{subsec:proofMain}

We apply the work in this section to $X=T$.  Since $T$ is assumed to have maximal Picard rank each $\omega_j$ can be approximated by rational K\"ahler classes so Proposition \ref{prop:Q'nefFnonpositive} applies.   We claim that:
\begin{enumerate}
\item $\mathcal R_{i,0}$ has the HR property over $V$ for $2\le i\le d-1$,
\item $\mathcal R_{d,0}|_W$ has the Hodge-Riemann property with respect to $h$.
\end{enumerate}
Since $\mathcal R_{d,0}(\alpha,\alpha') = \mathcal Q_{d}(\alpha,\alpha') = \int_X \alpha s_{\lambda}(\underline{\omega})\alpha$ for $\alpha,\alpha'\in H^{1,1}_\mathbb R(X)=W$, the second statement is precisely that $s_{\lambda}(\underline{\omega})$ has the Hodge-Riemann property, as claimed in Theorem \ref{thm:main}.

The facts established in this Section show on one hand that %Note now that by For the first statement, recall that $$\mathcal R_{i,0}' = (d-i+1) \mathcal R_{i-1,0}$$
the hypotheses in Theorem \ref{thm:recursion} are fulfilled for the sequence $(\mathcal R_{i,t})_{1\le i\le d-1}$ and on the other hand that  the family  $\mathcal R_{d,t}$ satisfies conditions (B1-B5). 
Thus Claim (1) holds by Theorem \ref{thm:recursion}. The special case $i=d-2$ of Claim (1) tells us now that the form  $\mathcal R''_{d,0} = 2\mathcal R_{d-2,0}$ has the Hodge-Riemann property. Hence the hypotheses of Theorem \ref{thm:augmentation2} are fulfilled for the family  $\mathcal R_{d,t}$ and thus $\mathcal R_{d,0}|_W$ has the Hodge-Riemann property with respect to $h$, which is (2) and finishes the proof of Theorem \ref{thm:main}.
%\end{proof}

\subsection{Proof of the linear case of Theorem \ref{thm:sketch}}
\label{subsec:proofSketch}
Let $d\ge 2$, $e\ge 1$ be two integers and $\lambda=(\lambda_{1},\ldots,\lambda_{N})$ be a partition of $d-2$ with $0\le\lambda_{N}\le\ldots\le\lambda_{1}\le e$. Let further $E$ be a $d$-dimensional complex vector space. 
We want to show that if $\omega_{1},\ldots,\omega_{e}$ are strictly positive $(1,1)$-forms on $E$ then $s_{\lambda}(\omega_{1},\ldots,\omega_{e})$ has the Hodge-Riemann property, i.e.  the quadratic form 
$$Q_{s_{\lambda}(\omega_{1},\ldots,\omega_{e})}:\bigwedge^{1,1}_\R E^*\to\R, \ \alpha\mapsto\frac{\alpha^{2}\wedge s_{\lambda}(\omega_{1},\ldots,\omega_{e})}{\vol}$$ has signature $(1,h^{1,1}-1)$. For this we choose a $d$-dimensional  complex torus $T$ with maximal Picard number,  say $T=(\C/\Z[i])^d$, and endow it with the flat hermitian metric coming from the  euclidean metric on $\C^d$. Then for $p,q\in\Z$ the harmonic $(p,q)$-forms on $T$ are precisely the parallel $(p,q)$-forms on $T$. Thus if we identify $E$ to the holomorphic tangent space of $T$ at some point $x\in X$, we get by Hodge theory natural isomorphisms 
$$\bigwedge^{p,q}_\R E^*\cong H^{p,q}_\R(T)$$
and our claim directly follows from Theorem \ref{thm:main}.

%%%%%%%%%%%%%%%%%%%%%%%%%%%%%%%%%

\section{Pointwise to global argument and the K\"ahler case of Theorem \ref{thm:sketch}}\label{sec:pointwise_to_global}
In this section we show that in our set-up the Hodge-Riemann property in the linear case implies the Hodge-Riemann property in the K\"ahler case. The argument follows the same line as the corresponding one of Gromov in \cite{Gromov90} but we give it in detail as parts of it go over to the non-K\"ahler case.  We end the section with a discussion on the notion of balanced metrics of Hodge-Riemann type on compact complex manifolds, Remark \ref{rem:HRtype}, which was introduced by Chen and Wentworth in their  recent paper \cite{CW_HR}. As an application of our main result we show in Corollary \ref{cor:HRstructure} that the Schur forms $s_{\lambda}(\omega_{1},\ldots,\omega_{e})$ appearing in our Theorem \ref{thm:sketch} give rise to such metrics and thus provide further situations where the results of \cite{CW_HR} apply.

Let $E$ be a $d$-dimensional complex vector space.   Inside $\bigwedge^{p,p}_\R E^*$ lies the closed convex cone $SP^{p,p}$ of {\em strongly positive forms}, which is by definition the convex hull of the set 
$\{i\alpha_{1}\wedge\bar\alpha_{1}\wedge\ldots\wedge i \alpha_{p}\wedge\bar\alpha_{p} \ | \ \alpha_{j}\in E^{*}, \ j\in\{1,\ldots,p\}\}$ of {\em simple forms}. Its dual cone is the cone of {\em weakly postive forms} $WP^{d-p,d-p}:=\{\eta\in  \bigwedge^{d-p,d-p}_\R E^* \ | \ \eta\wedge\gamma\ge0 \ \forall \gamma\in SP^{p,p}\}$. One also defines the cone of {\em positive forms} as 
$$P^{p,p}:= \{\eta\in  \bigwedge^{p,p}_\R E^* \ | \ \eta\wedge i^{(d-p)^{2}}\beta\wedge\bar\beta\ge0 \ \forall \beta\in \bigwedge^{d-p,0} E^* \}.$$ 
It is shown in \cite[Corollary 1.3]{HarveyKnapp} that the dual cone of $P^{p,p}$ is $P^{d-p,d-p}$. 
One sees easily that $SP^{p,p}\subset P^{p,p} \subset WP^{p,p}$ and that the interior of $SP^{p,p}$ in $\bigwedge^{p,p}_\R E^*$ is non-empty; in fact $\bigwedge^{p,p}_\R E^*$ admits a basis formed of simple forms \cite[Chapter III Lemma 1.4]{DemaillyBook}. We will call a $(p,p)$-form {\em strictly positive, respectively strictly weakly positive or strictly strongly positive} if it lies in the interior of $P^{p,p}$, respectively of $WP^{p,p}$ or of $SP^{p,p}$. 
For $p=1$ and for $p=d-1$ we have $SP^{p,p}=P^{p,p}=WP^{p,p}$, \cite[Chapter III Corollary 1.9]{DemaillyBook}, and we will call forms lying in these cones simply {\em positive}. 
All this terminology extends to differential forms on complex manifolds. Directly from the definition we see that an exterior product of strongly positive forms is strongly positive and that an exterior product of weakly positive forms is weakly positive if all factors  but one are strongly positive.
\begin{definition}
To a real $(d-2,d-2)$-form $\Omega$ on $E$ (and a  volume form $\vol$) we associate an {\em intersection form} on $\bigwedge^{1,1}_\R E^*$ by the formula
$$
Q_\Omega: \bigwedge^{1,1}_\R E^*\times \bigwedge^{1,1}_\R E^*\to\R, \ (\alpha,\beta)\mapsto\frac{\alpha\wedge\Omega\wedge\beta}{\vol}.
$$
We will say that $\Omega$ has the {\em Hodge-Riemann property} if the bilinear form $Q_\Omega$ has the Hodge-Riemann property. This property does not depend on the choice of the volume form. 

A differential $(d-2,d-2)$-form $\Omega$ on a $d$-dimensional complex manifold $X$ will be said to have the  {\em Hodge-Riemann property pointwise} if for each point $x\in X$ the form $\Omega(x)$ has the   Hodge-Riemann property on $\bigwedge^{1,1}T^*_xX$.
\end{definition}

Let $X$ be a compact complex manifold of dimension $d\ge2$ and let $\Omega$ be a strictly weakly positive $(d-2,d-2)$-form on $X$ such that $i\partial\bar\partial\Omega=0$. Let $\cE^{p,q}$  be the sheaf of germs of smooth $(p,q)$-forms  on $X$. We will write $\cE^{p,p}_{\R}$ for the subsheaf of real forms.
We will be interested in the real Bott-Chern cohomolgy group of bidegree $(1,1)$ on $X$.  In terms of forms it is defined as $H^{1,1}_{BC}(X,\R):=\{\eta\in\cE_{\R}^{1,1}(X) \ | \ \d\eta=0\}/i\partial\bar\partial\cE_{\R}^{0,0}(X).$
The following {\em intersection form} on $H^{1,1}_{BC}(X)$
$$Q_{[\Omega]}([\alpha]_{BC},[\beta]_{BC}):=\int_{X}\alpha\wedge\beta\wedge\Omega $$ is well defined. 
(In the  notation $Q_{[\Omega]}$ the square brackets around $\Omega$ are meant to suggest that the intersection form only depends on the Aeppli cohomology class of $\Omega$, a fact which we will not need in this paper.)

Consider further a strictly positive $i\partial\bar\partial$-closed $(d-1,d-1)$-form $\eta$ on $X$. Note that any compact complex manifold of dimension $d$ admits such forms; they are the $d-1$ powers of Gauduchon forms. Then the map 
$$L_{[\eta]}: H^{1,1}_{BC}(X)\to \R, \ [\alpha]_{BC}\mapsto\int_{X}\alpha\wedge\eta$$ is well defined. %We denote its kernel by $P^{1,1}_{\eta}(X)$. 

\begin{remark}\label{rem:non-kaehler}
If $X$ is K\"ahler then $L_{[\eta]}$ is clearly positive on K\"ahler classes, so  $L_{[\eta]}\ne0$ in this case. 
More generally, $L_{[\eta]}$ is positive on classes of non-zero positive $\d$-closed currents of type $(1,1)$. Such currents always exist when $\dim X=2$  surfaces, cf. \cite{LamariJMPA}, but this need not be the case in  higher dimensions. In fact there are  examples of compact complex manifolds % having no non-zero positive $\d$-closed currents of type $(1,1)$ and in fact 
allowing strictly positive $\d$-exact $(d-1,d-1)$-forms $\eta$, \cite{Yachou}, and clearly for such $(d-1,d-1)$-forms the corresponding  linear forms $L_{[\eta]}$  vanish identically.
\end{remark}

The following Lemma is certainly well-known to experts, but we give a proof for completeness.
\begin{lemma}\label{lem:fredholm}  Let $\eta$ be a strictly positive $i\partial\bar\partial$-closed  $(d-1,d-1)$-form on $X$  and let $\vol$ be a volume form on $X$. Then the image of the differential operator on the space of smooth real functions 
$$P:\cC^\infty(X)\to\cC^\infty(X), \ f\mapsto\frac{ i\partial\bar\partial f\wedge\eta}{\vol}$$
is the space $\{ g\in \cC^\infty(X) \ | \ \int_Xg\vol=0\}.$
\end{lemma}
\begin{proof} A direct computation shows that the symbol of the operator $P$ at a point $x\in X$ and a real cotangent vector in $T_{x}^{*}X$ projecting to  $\xi\in (T_{x}^{1,0}X)^{*}$ is $\frac{i\xi\wedge\bar\xi\wedge\eta(x)}{\vol(x)}\ne 0$, hence $P$ is an elliptic operator, \cite[Chapter VI Section 1]{DemaillyBook}.
 
On $\cC^\infty(X)$ we consider the scalar product
$$(f,g)\mapsto\langle f,g\rangle:=\int_{X}fg\vol$$
with respect to which the formal adjoint of $P$ is  
$$P^{*}(g):=\frac{ i\partial\bar\partial g\wedge\eta-i\bar\partial g\wedge\partial\eta+i\partial g\wedge\bar\partial\eta}{\vol}.$$
By the general theory of elliptic differential operators we know that the range of $P$ is closed and that there is an orthogonal direct sum decomposition $\cC^\infty(X)=P(\cC^\infty(X))\oplus\Ker P^{*}$, \cite[Chapter VI Corollary 2.4]{DemaillyBook}. Since $\eta$ is strictly positive we may apply the maximum principle of E. Hopf \cite[Chapter III theorem 1.10]{KobayashiVectorBundles} and obtain that $\Ker P^{*}$ consists of the constant functions alone. 
From this the Lemma follows.
\end{proof}

\begin{corollary}\label{cor:tofredholm}  Let $\eta$ be a strictly positive $i\partial\bar\partial$-closed  $(n-1,n-1)$-form and $\alpha$ be a $\d$-closed $(1,1)$-form on $X$ such that $\int_{X}\alpha\wedge\eta=0$. Then there
exists a smooth representative $\tilde\alpha$ in the Bott-Chern cohomology class $[\alpha]_{BC}$ of $\alpha$ such that
$$\tilde\alpha\wedge\eta=0.$$
\end{corollary}
\begin{proof}
Clearly it suffices to find a smooth function $f$ on $X$ such that $i\partial\bar\partial f\wedge\eta=-\alpha\wedge\eta$. By Lemma \ref{lem:fredholm} for any top degree form $\sigma$ on $X$ with $\int_X\sigma=0$ a smooth solution $f$  to the equation $$i\partial\bar\partial f\wedge\eta=\sigma$$ exists and we are done. 
\end{proof}

Take $\omega_{0}$ a strictly positive $(1,1)$-form on $X$ such that $\eta:=\Omega\wedge\omega_{0}$ is $\partial\bar\partial$-closed. Note that, by Gauduchon's \cite[Th\'eor\`eme I 14]{Gauduchon84} again, for any  strictly positive $(1,1)$-form $\omega$  on $X$ a positive function $f$ exists such that  $i\partial\bar\partial(f\omega\wedge \Omega)=0$.

%The main result of this section is the following
\begin{proposition}\label{prop:pointwisetoglobal}
If $\Omega$ has the Hodge-Riemann property pointwise and $\omega_{0}$ is a strictly positive $(1,1)$-form on $X$ such that $\Omega\wedge\omega_{0}$ is $\partial\bar\partial$-closed, then $Q_{[\Omega]}$ is negative definite on the subspace $\Ker L_{[\Omega\wedge\omega_{0}]}$ of $H^{1,1}_{BC}(X)$. In particular,  $Q_{[\Omega]}$ has the Hodge-Riemann property when $X$ is K\"ahler.
\end{proposition}
\begin{proof} Let 
$\alpha$ be a $\d$-closed $(1,1)$-form on $X$. By Corollary  \ref{cor:tofredholm} applied to $\eta:=\Omega\wedge\omega_{0}$ if the class $[\alpha]_{BC}$ belongs to $\Ker L_{[\Omega\wedge\omega_{0}]}$, there
exists a smooth representative $\tilde\alpha$ in $[\alpha]_{BC}$ such that
$$\tilde\alpha\wedge\Omega\wedge\omega_{0}=0.$$ Since $\omega_{0}$ is strictly positive it follows by the pointwise Hodge-Riemann property of $\Omega$ that $\tilde\alpha\wedge\tilde\alpha\wedge\Omega\le 0$ with equality if and only if $\tilde\alpha$ vanishes identically. Thus 
$$Q_{[\Omega]}([\alpha]_{BC},[\alpha]_{BC})=\int_{X}\tilde\alpha\wedge\tilde\alpha\wedge\Omega\le 0 $$
and equality holds if and only if 
$[\alpha]_{BC}=0$. Hence $Q$ is negative definite on $\Ker L_{[\Omega\wedge\omega_{0}]}$. If moreover $X$ is K\"ahler, $Q_{[\Omega]}$ will be positive on K\"ahler classes, hence $Q_{[\Omega]}$ will  have the Hodge-Riemann property as stated.
\end{proof}

Using this and the already established linear case of Theorem \ref{thm:sketch} we get the following Corollary which finishes the proof of Theorem \ref{thm:sketch}. 

\begin{corollary}\label{cor:KaehlerCase} The second statement of Theorem \ref{thm:sketch}, i.e. the K\"ahler case of that theorem, holds. 
\end{corollary}

\begin{example}[The Case of Surfaces]\label{ex:surfaces}
For $d=2$ the positive form $\Omega$ will be just a positive function. In this case it automatically has the Hodge-Riemann property  pointwise.   Since it is required to be also $\partial\bar\partial$-closed, $\Omega$ has to be a constant function. We will  assume that this constant is $1$. Then the corresponding intersection form on $H^{1,1}_{BC}(X)$  is $Q([\alpha]_{BC},[\beta]_{BC}):=\int_{X}\alpha\wedge\beta$. Consider now a Gauduchon form $\omega$ on $X$. 
By Proposition \ref{prop:pointwisetoglobal} $Q$ is negative definite on $\Ker L_{[\omega]}$. As remarked before we have $\Ker L_{[\omega]}\varsubsetneq H^{1,1}_{BC}(X)$ in the surface case. Two cases may occur:
\begin{enumerate}
\item If $X$ is K\"ahler we have seen that $Q$ has the Hodge-Riemann property and we thus recover the Hodge Index Theorem in this case.
\item If $X$ is non-K\"ahler and if $\alpha$ is a smooth representative of a non-zero $\d$-exact positive $(1,1)$-current, then $\alpha$ is $\d$-exact,  $L_{[\omega]}([\alpha]_{BC})>0$ and $Q([\alpha]_{BC}, [\beta]_{BC})=0$ for all $[\beta]_{BC}\in H^{1,1}_{BC}(X)$. Thus $Q$ is degenerate semi-definite  in this case.
\end{enumerate} 
\end{example}
Following \cite{CW_HR}
we introduce the following

\begin{definition}
\label{HRstructure} We say that a pair of strictly  weakly positive forms $(\Omega,\omega)$ of types $(d-2,d-2)$ and $(1,1)$ respectively on a $d$-dimensional complex manifold $X$ defines a   {\em Hodge-Riemann structure} on $X$ if $\Omega$ has the Hodge-Riemann property pointwise and 
is $\partial\bar\partial$-closed.
\end{definition}

\begin{remark}\label{rem:HRtype}

As mentioned before if $(\Omega,\omega)$ defines a Hodge-Riemann structure on the compact complex manifold $X$, there exists a unique conformal rescaling of $\omega$ to $\omega_{0}:=f\omega$ such that $i\partial\bar\partial(\omega_{0}\wedge\Omega)=0$. We denote the $d-1$-root of $\omega_{0}\wedge\Omega$ by $\omega'$. This is a Gauduchon form on $X$. Then the hermitian metric $\omega'$ is {\em balanced of Hodge-Riemann type} in the sense of \cite[Definition 2.7]{CW_HR} if  $\Omega$ and $\omega_{0}\wedge\Omega$ are moreover $\d$-closed and $\Omega$ is strictly positive. 
Note that the requirement that $\Omega$ be strictly positive is equivalent to the condition  appearing in \cite[Definition 2.7 (2)]{CW_HR} that $\Omega$ be a {\em Hodge-Riemann form for $(2,0)$}. Thus balanced hermitian metrics of Hodge-Riemann type in the sense of \cite[Definition 2.7]{CW_HR}  are particular cases of Hodge-Riemann structures on compact complex manifolds in our sense. It is worthwhile noticing that the results of \cite[Section 3]{CW_HR} on polystable vector bundles are more generally valid over compact complex manifolds carrying Hodge-Riemann structures. (In order to see this it suffices to apply \cite[Lemma 2.1.5]{LuTe} in the proof of \cite[Theorem 3.2]{CW_HR}.)   
\end{remark}

\begin{example}\label{ex:astheno} A hermitian metric (which we identify with its fundamental $(1,1)$-form) $\omega$ on a $d$-dimensional compact complex manifold $X$ is called {\em astheno-K\"ahler} if $i\partial\bar\partial\omega^{d-2}=0$. This notion was introduced by Jost and Yau in \cite{JostYau}. Examples of non-K\"ahler astheno-K\"ahler metrics have been constructed for all $d\ge 3$, cf. \cite{FinoTomassini},  \cite{LatorreUgarte}, . If $\omega$ is an astheno-K\"ahler metric on $X$ then clearly $(\omega^{d-2},\omega)$ defines a Hodge-Riemann structure on $X$. %\todo{Jost and Yau prove Siu's Rigidity Lemma for astheno-kaehler manifolds. Chen and Wentworth prove it for balanced manifolds of HR type. Check whether it holds more generally for manifolds with HR structure.}
\end{example}

The main example of Hodge-Riemann structures provided by our paper occurs however in the K\"ahler set-up and is the content of the following corollary to our main theorem.

\begin{corollary}\label{cor:HRstructure} Let $\omega_{1},\ldots,\omega_{e}$ be K\"ahler forms on a compact complex manifold of dimension $d\ge 3$, let $\omega$ be an arbitrary strictly positive $(1,1)$-form on $X$ and let $\lambda=(\lambda_{1},\ldots,\lambda_{N})$ be a partition of $d-2$ with $0\le\lambda_{N}\le\ldots\le\lambda_{1}\le e$. . Then $$(s_{\lambda}(\omega_{1},\ldots,\omega_{e}),\omega)$$ defines a Hodge-Riemann structure on $X$ and the quadratic form $[\alpha]_{BC}\mapsto\int_{X}\alpha^{2}\wedge s_{\lambda}(\omega_{1},\ldots,\omega_{e})$ has the Hodge-Riemann property on $H^{1,1}_{BC}(X)$. If $\omega$ is moreover closed then $(s_{\lambda}(\omega_{1},\ldots,\omega_{e}),\omega)$ induces a balanced hermitian metric of Hodge-Riemann type on $X$ as in Remark \ref{rem:HRtype}.
\end{corollary}

%%%%%%%%%%%%%%%%%
%%%%%%%%%%%%%%%%%5

\section{Convex combinations of Schur classes} \label{sec:convex}
In this section we show that the device we developed to prove our main result also applies to prove the Hodge-Riemann property for certain convex combinations of Schur classes. The main criterion to decide for which convex combinations this happens, Theorem \ref{thm:convex}, is given in terms of irreducibility of certain cycles on appropriate Grassmann varieties.

To describe this, fix 
 integers $b\ge 1$, $e\ge 1$ and $d\ge b$. 
 We will consider the $(k-1)$-simplex $$\Sigma_{k-1}:=\{ x\in(\R_{\ge0})^k  \ | \ \sum_{i=1}^k x_i=1 \},$$ where $k:=k(b,e)$ is the number of partitions   $\lambda=(\lambda_{1},\ldots,\lambda_{n})$ of $b$ with $e\ge\lambda_{1}\ge \ldots\ge\lambda_{n}> 0$. 
 We will order these partitions in some arbitrary but fixed way as $\lambda^{(1)},\ldots,\lambda^{(k)}$.  
Then when $b=d-2$ we will be interested in finding conditions under which convex combinations of Schur classes $s_{\lambda^{(i)}}$, $1\le i\le k$, of ample vector bundles have the Hodge-Riemann property. (We exclude the case when $d=2$ and $b=0$  since it is trivial.)

We next introduce two definitions, one concerning linear sums of Schur polynomials and then one concerning linear sums of Schubert varieties. For $x\in  \Sigma_{k-1}$ define $$\Gamma_x:=\sum_{i=1}^k x_i s_{\lambda^{(i)}}.$$

By Fulton-Lazarsfeld \cite{FultonLazarsfeld}, for any $x\in  \Sigma_{k-1}$, for any projective $d$-dimensional manifold $X$, for any ample vector bundle $E$ of rank $e$ and for any ample class $h$ on $X$ it holds that $\int_X \Gamma_x(E) h^{d-b}>0$.

\begin{definition}
 We say that for $x\in  \Sigma_{k(d-2,e)-1}$ the characteristic class $\Gamma_x$  has the \emph{universal Hodge-Riemann property} if for any $d$-dimensional projective manifold $X$ and for any ample rank $e$ vector bundle $E$  on $X$ the class $\Gamma_x(E)$ has the Hodge-Riemann property. 
\end{definition}

Now to each partition $\lambda$ and integer $N\ge b$ we let $C_{\lambda}$ be the corresponding Schubert variety in $\Grass(N,\mathbb C^{N+e})$, where by $\Grass(N,\mathbb C^{N+e})$ we denote the Grassmannian of $N$-dimensional linear subspaces of $\mathbb C^{N+e}$.  That is, if $\lambda=(\lambda_{1},\ldots,\lambda_{n})$ is a partition of $b$ with $e\ge\lambda_{1}\ge \ldots\ge\lambda_{n}> 0$, we fix a sequence of subspaces $0\subsetneq A_1\subsetneq A_{2} \subsetneq\cdots \subsetneq A_n\subset \mathbb C^{N+e}$ with $\dim(A_i) = a_i: = e + i -\lambda_i$ for $1\le i\le n$ we set
$$C_\lambda : = \{ L\in \Grass(N,\mathbb C^{N+e}) \ | \ \dim L\cap A_i \ge i \text{ for all } i=1,\ldots, n \}.$$

\begin{definition} 
\label{def:good} We say that a point $x\in \Sigma_{k(b,e)-1}\cap \Q^{k(b,e)} $ is {\em irreducibly representable} if there exist positive  integers $N\ge b$ and $m$  such that the cycle $$m\sum_{i=1}^{k(b,e)} x_i C_{{\lambda^{(i)}}}$$ in $\Grass(N,\mathbb C^{N+e})$ is algebraically equivalent to an irreducible cycle. We denote the set of irreducibly representable points of $\Sigma_{k(b,e)-1}$ by $B_{b,e}$.
\end{definition}

\begin{theorem}\label{thm:convex}
For every $x\in \overline{B_{d-2,e}}$ we have 
\begin{enumerate}
\item The class $\Gamma_x$ %$\sum_{i=1}^k x_i s_{\lambda^{(i)}}$ 
has the universal Hodge-Riemann property. 
\item The class $\Gamma_x$ %$\sum_{i=1}^k x_i s_{\lambda^{(i)}}$ 
has the Hodge-Riemann property in the linear case.
\item The class $\Gamma_x$ %$\sum_{i=1}^k x_i s_{\lambda^{(i)}}$ 
has the Hodge-Riemann property in the K\"ahler case.
\end{enumerate}
\end{theorem}

Before giving the proof we will describe the connection between Schubert varieties and Schur classes of vector bundles through some aspects of invariant algebraic cycles on projective manifolds under a group action. 

Fix two positive integers $e$ and $N$ and two complex vector spaces $H$ and $H'$ of dimensions $e+N$ and $e$ respectively. Let $G:=\GL(H')$ be the general linear group acting on $H'$. Then $G$ acts by composition on the vector space $\Hom(H,H')$ and further on its projectivization $\P_{sub}(\Hom(H,H'))$. 
The open subset $\Hom(H,H')^\circ$ consisting of surjective homomorphisms is $G$-invariant and its quotient by the $G$ action is given by the map $p:\Hom(H,H')^\circ\to\Grass(N,H)$, $\sigma\mapsto\Ker(\sigma)$. (Actually  $\Hom(H,H')^\circ$ is the subset of (semi-)stable points for the $G$-action on $\Hom(H,H')$).

Let now $C$ be an algebraic cycle on $\Grass(N,H)$ and $\hat{C}$ be the algebraic cycle on $\P_{sub}(\Hom(H,H'))$ given by the closure  of $p^{-1}(C)$ in $\Hom(H,H')$. This is a $G$-invariant cycle on $\P_{sub}(\Hom(H,H'))$.  
More generally, we claim that any algebraic family $\cC$ of cycles on $\Grass(N,H)$ induces in this way an algebraic family $\hat{\cC}$ of $G$-invariant cycles on  $\P_{sub}(\Hom(H,H'))$. Since on Grassmannians algebraic equivalence and rational equivalence of cycles coincide, we may as well work with the latter.

This construction may be extended to a relative situation as follows. 
If $E$ is a rank $e$ vector bundle on a projective manifold $X$ and $F: = H^*\otimes E=\cH om(H,E)$ then any algebraic cycle $C$ on $\Grass(N,H)$ induces an algebraic cycle $\hat{C}_E$ on $\P_{sub}(\cH om(H,E))=\P_{sub}(F)$ and similarly for algebraic families of cycles on $\Grass(N,H)$.
It follows that two algebraic equivalent cycles $C$, $C'$ on $\Grass(N,H)$ induce algebraic equivalent cycles $\hat{C}_E$, $\hat{C'}_E$ on $\P_{sub}(F)$. Thus we have a morphism of Chow groups 
\begin{equation}
A_*(\Grass(N,H))\to A_*(\P_{sub}(\cH om(H,E))) \quad C\mapsto\hat{C}_E\label{eq:morphismchow}\end{equation}
 Note that by construction if $C$ is irreducible, then also $\hat C$ and $\hat{C}_E$ will be irreducible.

In particular we can apply this when  $C= C_{\lambda}$ is a Schubert variety.   Then the closure  of $p^{-1}(C)$ in $\Hom(H,H')$ coincides with the subvariety
$$\tilde C : = \{ \sigma\in \Hom(H,H'): \dim \ker(\sigma(x))\cap A_i \ge i \text{ for all } i=1,\ldots, n\}$$
of $\Hom(H,H')$. Indeed, $p^{-1}(C)=\tilde C \cap \Hom(H,H')^\circ$, $\tilde C$ is closed in $\Hom(H,H')$ and  $\tilde C$ is irreducible since $C$ is irreducible.%   (cf. \cite[Lemma A.7.2]{FultonIT}). 
 
 \begin{remark}
The reader can check that for Schubert cycles $C$ in $\Grass(N,H)$ and for vector bundles $E$ of rank $e$ on some projective manifold $X$, the construction of the cycles $\hat{C}_E$ on $\P_{sub}(\cH om(H,E))$  recovers the definition $\hat{C}_E$ from Section \ref{subsec:geometric}.
\end{remark}

\begin{proof}[Proof of Theorem \ref{thm:convex}]
{\bf (1)}
Let $X$ be a complex projective manifold of dimension $d$, $E$ be an ample vector bundle of rank $e$ on $X$ and $x$ an element in  $\overline{B_{d-2,e}}$. We fix an ample class $h$  on $X$ and will show that $\Gamma_x(E)$ has the Hodge-Riemann property with respect to $h$.

We make the same notations as in Section \ref{sec:cone_classes}, 
$W:=H^{1,1}_\R(X)$, 
$\hat{X} : = X\times \mathbb P^d$,  $\zeta:=c_1(\cO_{\mathbb P^d}(1))_\R\in H^{1,1}_\R(\mathbb P^d)$,  $V:= H^{1,1}_\R(\hat{X}) = W \oplus \mathbb R{\zeta}$ and define $\cQ_i$ and $\cR_{i,t}$ by imitating Definitions \ref{def:definitionQ_i} and \ref{def:definitionR_i} but for the class $\Gamma_x$ instead of $s_\lambda$, i.e.
\begin{align*}
\mathcal Q_i(\beta,\beta'):=\mathcal Q_i(\beta,\beta';x):= \int_{\hat{X}} \beta 
\Gamma_x(E\langle\zeta\rangle)\zeta^i h^{d-i} \beta'\end{align*}
for $\beta,\beta'\in V$ and 
$$\cR_{i,t}:=\sum_{k=0}^i \binom{d-i+k}{k}t^k\cQ_{i-k}$$
 for $t\in\R$ and  for $0\le i\le d$.
 For $i$ outside the given range we set $\cQ_i=0$ and $\cR_i=0$.
 Here to simplify notation we have used the symbols $E$ and $\zeta$ also for their pull-backs to $\hat X$. 
 
 We will apply our algorithm as in Section \ref{subsec:proofMain} to show that $\mathcal R_{d,0}|_W$ has the Hodge-Riemann property with respect to $h$ which is exactly statement (1) of Theorem \ref{thm:convex}. 
 For this we need to check properties (A) and (B) for $\cR_{i}$ along the lines of Section \ref{subsec:propertiesR_i}.

The only delicate part is establishing an analogue of Proposition  \ref{prop:Q'nefFnonpositive}. As both statements in that Proposition are closed under variation of $x$, we may assume that $x$ lies in the set $B_{d-2,e}$ of irreducibly representable points of $\Sigma_{k(d-2,e)-1}$. Now for $x$ in this set, the idea is to have a geometric interpretation of the forms $\cR_{i,t}$   as forms $\cS_{i,t}$ in the set-up of Section \ref{subsec:geometric}, where this time $\cS_{i,t}$ is computed by integrating on a suitable irreducible variety $\widehat{C(x)}_E$ (that now depends on $x$).   This may be achieved since $x$ was assumed to be in 
$B_{d-2,e}$, so we can choose $N$ and $m$ as in Definition \ref{def:good}, $H$ a complex  vector space of dimension $e+N$ and choose $C(x)$ to be an irreducible cycle algebraically equivalent to  
$m\sum_{i=1}^{k(d-2,e)} x_i C_{{\lambda^{(i)}}}$ in $\Grass(N,H)$ and then take $\widehat{C(x)}_E$ to be the corresponding cone in $\P_{sub}(\cH om(H,E))$ given in \eqref{eq:morphismchow}.  Then as $C(x)$ is irreducible, so is $\widehat{C(x)}_E$.

Then similarly to Proposition \ref{prop:pushforwardschur} the forms $m\cR_{i,t}$ correspond to the forms 
\begin{align*}\cS_{i,t}(\beta,\beta')&:=\int_{\widehat{C(x)}_E}\beta c_{f-(d-i)}(U\langle \pi^*(th)\rangle)\pi^*(h)^{d-i}\beta'\\&=m\sum_{j=1}^{k(d-2,e)} x_j \int_{\widehat{C_{{\lambda^{(j)}}}}_E}\beta c_{f-(d-i)}(U\langle \pi^*(th)\rangle)\pi^*(h)^{d-i}\beta'\end{align*}
on $H^{1,1}_\R(\mathbb P_{\sub}(\Hom(H,E)))$, where as in Section \ref{subsec:geometric} $f  = \rank(\Hom(H,E))-1 = e(e+N)-1$  and $U$ is the universal quotient bundle on  $\mathbb P_{\sub}(\Hom(H,E))$.
Now the same arguments used in Proposition  \ref{prop:Q'nefFnonpositive} will work and prove the two corresponding statements in the present case too (we remark that irreducibility of $\widehat{C(x)}_E$ is crucial here, since in that argument we use a resolution of singularities to pass to the smooth case).

{\bf (2)}
The proof of this statement goes exactly as in Section \ref{subsec:proofSketch} by reduction first to the K\"ahler case on tori with maximal Picard number and in that case as in Section \ref{subsec:proofMain} by the above argument.

{\bf (3)} is a consequence of the linear case (2) and of Proposition \ref{prop:pointwisetoglobal}.
\end{proof}

\begin{corollary}\label{cor:closure} In general there exist  rational points in $\Sigma_{k(b,e)-1}\setminus \overline{B_{b,e}}$. \end{corollary}

\begin{proof}
By Theorem \ref{thm:convex} it suffices to show that $d$, $e$ and $x\in \Sigma_{k(d-2,e)-1}\cap\Q^{k(d-2,e)}$ exist together with a projective manifold $X$ of dimension $d$ and an ample vector bundle of rank $e$ on $X$ such that $\Gamma_x(E)$ does not have the Hodge-Riemann property. This was done in \cite[Section 9.2]{RossToma} for $d=5$ and $e=3$.  
\end{proof}

%%%%%%%%%%%%%%%%%%%%%%%%%%%%%%%%%%%%%%%%%%%%%%

%%%%%%%%%%%%%%

%\printbibliography

\bibliography{hr_linear}
\bibliographystyle{amsalpha}

\providecommand{\bysame}{\leavevmode\hbox to3em{\hrulefill}\thinspace}
\providecommand{\MR}{\relax\ifhmode\unskip\space\fi MR }
\providecommand{\MRhref}[2]{%
  \href{http://www.ams.org/mathscinet-getitem?mr=#1}{#2}
}
\providecommand{\href}[2]{#2}

\end{document}